\documentclass[11pt]{article}

\usepackage[dvips]{graphics}
\usepackage{amsmath,amsthm,amsopn,amssymb,epsfig,psfrag,color}
\usepackage{multicol}
\usepackage{amsfonts}
\usepackage[english]{babel}
\usepackage{multirow}
\usepackage{epstopdf}
\usepackage[title,titletoc,toc]{appendix}
\usepackage[right, pagewise]{lineno}
\usepackage{graphicx}
\usepackage{subfigure}
\usepackage[numbers, sort&compress]{natbib}
\usepackage{slashbox}

 \newcommand{\D}{\displaystyle}

\newtheorem{Theorem}{Theorem}[section]
\newtheorem{Lemma}{Lemma}[section]
\newtheorem{Proposition}{Proposition}[section]

\newtheorem{Remark}{Remark}[section]

\numberwithin{equation}{section} \numberwithin{figure}{section}

\def\D{\displaystyle}

\begin{document}

\title{Enhancing Population Persistence by a Protection Zone in a
Reaction-Diffusion Model with Strong Allee Effect}
\author{Yu Jin\footnote{
Department of Mathematics,
University of Nebraska-Lincoln,
Lincoln, NE, 68588, USA. Email: yjin6@unl.edu}, \,\,
Rui Peng\footnote{School of Mathematics and Statistics, Jiangsu Normal University, Xuzhou, 221116, Jiangsu, China.
Email: pengrui\_seu@163.com},\,\,
Jinfeng Wang\footnote{School of Mathematical Sciences and Y.Y. Tseng Functional Analysis Research Center, Harbin
Normal University, Harbin 150001, Heilongjiang, China. Email: jinfengwangmath@163.com
}
}
\date{}
\maketitle

\begin{abstract}
Protecting endangered species has been an important issue in ecology. We derive a reaction-diffusion model for a population in a one-dimensional bounded habitat, where the population is subjected to a strong Allee effect in its natural domain but obeys a logistic growth in a protection zone. We establish the conditions for population persistence and extinction via the principal eigenvalue of an associated eigenvalue problem and investigate the dependence of this principal eigenvalue on the location (i.e., the starting point and the length) of the protection zone. The results are used to design the optimal protection zone under different boundary conditions, that is, to suggest the starting point and length of the protection zone with respect to different population growth rate in the protection zone, in order for the population to persist in a long term.

\end{abstract}



\textcolor{red}{}

\section{Introduction and our model}

\subsection{Introduction}

Endangered species or native species may be subjected to density decline due to various reasons such as climate change, habitat change, pollution,  predation, harvest, invasion of alien species, etc. Strategies or management plans have been designed in order to  maintain the species  and their habitats at a certain level so that the populations can continue to grow and influence the habitats and the associated ecosystems in a healthy manner. Establishing protection zones such as natural reserves has been considered as an effective method to
protect endangered species from extinction, or otherwise to slow
down the speed of its extinction; see e.g.,
\cite{Holt1977,Holt1984,Ami2005,Jerry2010,Loisel,Halpern}.

The Allee effect is a ubiquitous phenomenon in biology showing that the
population growth rate can be significantly small or even negative
when the population density or size is very low, due to e.g., the
difficulty in finding mates or defending against predators, and that
the population growth rate is also negative when the population density
or size exceeds a large number due to  the environment limit; see e.g.,
\cite{Allee1931,Courchamp,Gascoigne}.
In particular, it is referred to a strong Allee effect if the growth rate is negatively related to the population size or a weak Allee effect if the growth rate is positively related to the population size, when the population size is small.
This phenomenon has been observed, for instance,
for the west Atlantic cod (Ganus morhus)
(see e.g., \cite{Deroos}), Vancouver Island marmot (Marmota
vancouverensis) (see e.g., \cite{Brashares}), and
 marine species such as blue crab, oyster (see e.g., \cite{Gascoigne}).
The Allee effect has been increasingly studied in population
models in recent decades. The dynamics of a single species with Allee effect has been explored in e.g., \cite{Sunhan2020,Sun2020,Dupeng2019,Wangshi2020,Keitt,Wangshi2019-1,Wangshi2019-2,Shishivaji2006}; the dynamics of predator-prey models with Allee effect on the prey has been investigated in e.g., \cite{Cui2014,wang2016,Wangshiwei2011,Wangshiwei20111}.

 When a species is subjected to significant predation, harvest, or competition, establishing  protection zones, where its population can enter or exit freely but its predators, competitors, or human beings are prohibited from entering, allows the species to freely grow and be preserved within the protection zones. Population models have been developed to predict and evaluate population responses to management actions such as protection zones that affect growth, survival, and spread.
  Reaction-diffusion predator-prey models and competition models with protection zones for the prey or for the weak competitor were first investigated in a sequence of works by Du et al, in \cite{Dushi2006} for
 a Holling type II predator-prey system, in \cite{Duliang2008,GaoLiang} for the competition system, in \cite{Dupengwang2009} for a Leslie type predator-prey system, and in \cite{Dushi20062}
 for  predator-prey
systems with protection coefficients. Results there showed that the prey or the weak competitor can successfully survive with the other species if the size of the protection zone is large enough while situations may be complicated and the persistence of the species may depend on demographic and dispersal characteristics of all involved species as well as the habitat features if the size of the protection zone is small.
The effect of cross-diffusion  on the stationary problem (i.e., the problem for the steady state) was studied for  predator-prey systems with  protection zones in \cite{Oeda2011,Li2017} and for
a Lotka-Voltera competition system with a protection
zone in \cite{Wang2013}.
Predator-prey systems with protection zones for the prey were also studied in \cite{he2016,Tripathi2015} (with Beddington-DeAngelis type functional response), in \cite{Chang2013} (with a prey refuge where a constant amount of preys are protected from being predated),
and in \cite{Cui2014} (with
a protection zone and strong Allee effect for the prey). For a single species living in a habitat where a protection zone is set up in a natural domain where the population is subjected
to extra removal due to external reasons, its dynamics was studied in e.g., \cite{Fan2001,Zou2011} (in a temporally constant environment), in \cite{Dieudu2016} (in
random environments), and in \cite{Cui2017} (with spatially
heterogeneous harvesting quota). In particular,
when a single species was subjected to strong
Allee effect outside of the protection zones, reaction-diffusion models were applied to investigate the role of the length
and the structure (one single connected patch or separated patches)
of the protection zones on species spreading in \cite{Dupeng2019}, as well as
the effect of protection zones on the
long-term behaviors and spread of a species with free boundary conditions in
\cite{Sunhan2020,Sun2020}.

In this work, we are interested in the effect of the protection zone on the long-term dynamics of an endangered species with spatial dispersal and how the protection zone can enhance population persistence of the species. Our main purpose is to suggest strategies for the placement of an optimal protection zone in order for the population to persist when the initial population density is small.
Note that it may be very difficult for an endangered species to recover when its density is sufficiently low. We assume that the population is subjected to a strong Allee effect in its natural domain so that such a species will die out if its initial population size is small. To help the population persist,  a protection zone is set up in its domain and in this zone the population can keep growing to its carrying capacity.
We will study a reaction-diffusion system that describes the dynamics of the population in such a habitat. Previous studies have discovered that a large size of the protection zone can usually help population persist (see e.g., \cite{Dushi2006,Duliang2008,Dupengwang2009,Fan2001,Zou2011}). We will establish  population persistence and extinction via the principal eigenvalue of the eigenvalue problem associated with the linearized system at the trivial solution. Then we will derive the precise influence of the protection zone, specifically, the starting point and the length of the protection zone, on the persistence conditions under different boundary conditions via the investigation of the dependence of the principal eigenvalue on these factors. Based on these results, we will then design the optimal location of the protection zone in each case.

\subsection{Our model}

To describe the spatiotemporal evolution of a single species that is subjected to the strong Allee effect growth, we consider a reaction-diffusion equation on a one-dimensional bounded habitat $(0,L)$:
\begin{equation}\label{alleesingle}
u_t=u_{xx}+g(u),\quad 0<x<L,\;\;t>0,
\end{equation}
where $u(t,x)$ is the density of the population at location $x$ at time $t$, $L$ is the total length of the habitat represented by the interval $[0,L]$. The growth rate $g(u)$ satisfies the following basic technical assumptions, which are similar to those in \cite{CosnerContrall,Shishivaji2006}:
\begin{enumerate}
 \item [{\bf (A1)}] $g\in C^1([0,+\infty))$, $g'(0)<0 =g(0)=g(a)=g(1)$ with $0<a<1$, $g(u)$ is positive for
  $a<u<1$, and $g(u)$ is negative otherwise.
 \end{enumerate}
Here the local carrying capacity of $g$ is rescaled as $1$, and $a$ is the local threshold value for the extinction/persistence of
the population, which is also known as the sparsity constant. The most important feature of \eqref{alleesingle} is that
 it admits bistable steady states, and different initial conditions can lead to different asymptotic
behaviors of the solutions \cite{Wangshi2019-2}. Especially, $\lim\limits_{t\rightarrow\infty}u(t,x)=0$ provided that the initial quantity is below  $a$.
   Hence, the persistence is always conditional.

Another universal acceptable growth for a single species is the logistic type, which leads to the following problem:
\begin{equation}\label{logsitcsingle}
u_t=u_{xx}+f(u),\quad 0<x<L,\;\;t>0,
\end{equation}
where the growth rate $f(u)$ satisfies the following basic technical assumptions, which are similar to those in \cite{Shishivaji2006}:
\begin{enumerate}
 \item [{\bf (A2)}] $f\in C^1([0,+\infty))$, $f'(0)>0 =f(0)=f(1)$, $f(u)$ is positive for
  $0<u<1$, and $f(u)$ is negative otherwise.
 \end{enumerate}
For simplicity, the local carrying capacity of $f$ is also rescaled as $1$.
 It is well known that all solutions of problem \eqref{logsitcsingle}
with non-negative and not identically zero initial data will converge to the steady state $u=1$ as $t\rightarrow +\infty$.
It leads to an unconditional persistence of the population for all initial condition.

Typical examples of strong Allee effect growth functions and  logistic growth functions are
\begin{equation}\label{growthfunex}
g(u)=ru(1-u)(u-a),\quad f(u)=ru(1-u),
\end{equation}
 where $r>0$ represents the intrinsic growth rate of the population, and $a\in (0,1)$ indicates that the population growth rate becomes negative when its density is below $a$.

To conserve an endangered species in a bounded region, one effective way is to introduce a protection zone within which the population's continuous growth is guaranteed. Mathematically, one may use the strong Allee effect growth to describe the dynamics of the endangered species in its natural domain, and a logistic type growth for its dynamics in the protection zone. As a result, we are led to consider the following model that describes the dynamics of a population in a one-dimensional habitat with a protection zone:
\begin{equation}\label{model1dge}
\begin{cases}
u_t=u_{xx}+F(u),&\, x\in (0,L),\ \ t>0,\\
u(0,x)=u_0(x)\geq 0, &\, x\in (0,L),
\end{cases}
\end{equation}
where  $$F(u)=\left\{
\begin{array}{l}
f(u), \quad x\in (\alpha,\alpha+l), \\
g(u),  \quad  x\in (0,\alpha)\cup (\alpha+l,L).
\end{array}
\right.$$
Here, $(\alpha,\alpha+l)$ represents a protection zone with starting point $\alpha$ and length $l>0$ satisfying $0\leq\alpha<\alpha+l\leq L$.
Note that the species follows the strong Allee effect growth $g(u)$ on $(0,\alpha)\cup (\alpha+l,L)$, while it obeys the logistic growth $f(u)$ in the protection zone $(\alpha,\alpha+l)$, and that $g $ and $f$ satisfy the conditions in {\bf (A1)} and {\bf (A2)}, respectively. We will use general functions $g$ and $ f$ to derive our theories, but use specific examples
of $g$ and $f$ in (\ref{growthfunex}) in numerical simulations.

Boundary conditions  at $x=0$ and $x=L$ can be assumed to be a general Robin type:
\begin{subequations}\label{udbcs}
\begin{align}
  \alpha_1 u_x(t,0)-\alpha_2u(t,0)=0, \quad\alpha_1\geq 0, \quad \alpha_2\geq 0, \quad\alpha_1^2+\alpha_2^2\neq 0, \quad t>0, \label{udbcsup}\\
    \beta_1 u_x(t,L)+\beta_2u(t,L)=0, \quad
 \beta_1\geq 0, \quad \beta_2\geq 0, \quad \beta_1^2+\beta_2^2\neq 0, \quad t>0,\label{udbcsdown}
\end{align}
\end{subequations}
which typically includes special boundary conditions such as
\begin{equation}\label{udbcsZFDD}
\begin{array}{lllll}
\text{Neumann condition: }  &u_x(t,0)=0 &\text{ or } &u_x(t,L)=0,& t>0,\\
\text{Dirichlet condition: } & u(t,0)=0 &\text{ or } & u(t,L)=0, & t>0.
\end{array}
\end{equation}

For any given $u_0\in C([0,L])$, it is easy to check that
equation \eqref{model1dge}-\eqref{udbcs} admits a unique
solution $u(\cdot,\cdot)\in
C^{1,2}((0,\infty)\times([0,L]\backslash\{\alpha,\alpha+l\}))\cap
C^{\frac{\gamma}{2},1+\gamma}((0,\infty)\times[0,L])
\cap C([0,\infty)\times [0,L])$ for any $\gamma\in(0,1)$
and $u(t,\cdot)$ exists for all $t>0$, and that if $u_0\geq 0,\not\equiv 0$, then $ u(t,x)>0$ for all $(t,x)\in (0,\infty)\times [0,L]$ except at the Dirichlet boundary ends; see e.g.,
\cite{JinPengShi,von1988}. Moreover, a simple comparison analysis gives $\limsup\limits_{t\rightarrow \infty} u(t,x)\leq 1$ for any nonnegative solution.

Given $t>0$, since any nonnegative solution $u(t,x)$ of \eqref{model1dge}-\eqref{udbcs} is continuously differentiable at the interface points $\alpha$ and $\alpha+l$, this means biologically that
the population density is continuous and the population flux is conserved at these points, where the population growth conditions change.

The rest of the  paper is organized as follows. In section 2, we will present general
theories for population persistence and extinction by virtue of the
principal eigenvalue of the associated eigenvalue problem for the linearized equation at the trivial solution. In
section 3, we will study the dependence of the principal eigenvalue
on the starting point and the length of the protection zone under different sets of boundary
conditions and then provide management plans for setting up the protection zone. Section 4 ends the paper with a brief discussion of the biological implications of the obtained results and some perspective of future work. In the appendix, we present the detailed calculation of Lemmas \ref{lambdaalphataneqnsgeq0} and \ref{lambdaalphataneqnsgg0}.

\section{The theories for population persistence and extinction}

In this section, we will establish theories for population persistence and extinction for (\ref{model1dge})-(\ref{udbcs}).
We first introduce some function spaces. In the case of
Neumann or Robin boundary
conditions,  let $\mathbb{X} = C^1([0,L], \mathbb{R})$  denote the Banach space of continuously differentiable functions on
 $[0,L]$ with the norm
$\textstyle\|u\|=\max\limits_{x\in[0,L]} |u(x) |+\max\limits_{x\in[0,L]} |u'(x) |$ for  $u\in \mathbb{X}$. The
set of nonnegative functions forms a solid cone $\mathbb{X}_+$ in  $\mathbb{X}$ with
interior $Int(\mathbb{X}_+) =\{ u\in \mathbb{X} : u(x) > 0, \forall x\in [0,L]\}$.
In the case of
  Dirichlet boundary conditions, let $\mathbb{X}=C_0([0,L], \mathbb{R}) \cap C^1([0,L],\mathbb{R})$ be the set of continuously differentiable functions on $[0,L]$ vanishing on
the boundary with
 the norm $\textstyle||u||=\max\limits_{x\in[0,L]} |u(x) |+\max\limits_{x\in[0,L]} |u^\prime(x) |$. The set of nonnegative functions $\mathbb{X}_+$ forms a solid cone in $X$ with nonempty interior in $X$ given by $Int(\mathbb{X}_+)=\{u\in \mathbb{X}_+: u(x)>0 \mbox{ for all } x\in(0,L), u_x(0)>0, u_x(L)<0\}$. If one of the boundary conditions is Neumann or Robin and the other one is Dirichlet, the function space $\mathbb{X}$ can be similarly defined.

 The linearized problem of (\ref{model1dge})-(\ref{udbcs}) at $u=0$ satisfies
\begin{equation}\label{model1dgelin}
\left\{
\begin{array}{l}
\D u_t=u_{xx}-H(x)u, \quad x\in (0,L),\quad t>0,\\
\text{ with boundary conditions as in (\ref{udbcs})},
\end{array}
\right.
\end{equation}
where
\begin{equation}\label{lineargrowthfun}
H(x)=\begin{cases} -f'(0),& x\in(\alpha,\alpha+l),\\
-g'(0),& x\in(0,\alpha)\cup(\alpha+l,L).
\end{cases}
\end{equation}
Letting $u(t,x)=e^{-\lambda t}\varphi(x)$, we obtain the eigenvalue problem associated with (\ref{model1dgelin}):
\begin{equation}\label{model1dgelineig0}
\left\{
\begin{array}{ll}
-\varphi_{xx}+H(x)\varphi=\lambda \varphi, &\quad x\in (0,L),\\
\alpha_1\varphi'(0)-\alpha_2 \varphi(0)=0,
&\quad
\beta_1\varphi'(L)+\beta_2 \varphi(L)=0.
\end{array}
\right.
\end{equation}
It is well-known that \eqref{model1dgelineig0} admits a unique principal eigenvalue,
denoted by $\lambda_1(\alpha,l)$, associated with a positive
eigenfunction in $Int (X_+)$, denoted by $\varphi_1$. Indeed, it follows from standard elliptic regularity theory that $\varphi_1\in C^2([0,L]\backslash\{\alpha,\,\alpha+l\})\cap C^1([0,L])$.

We now use the principal eigenvalue $\lambda_1(\alpha,l)$ of (\ref{model1dgelineig0}) to establish population persistence or extinction for \eqref{model1dge}-\eqref{udbcs}.
The main result in this section is given as below.
\begin{Theorem}\label{persistenceupt}
{\rm
The following statements hold.
\begin{enumerate}
  \item[(i)] If $\lambda_1(\alpha,l)<0$, then the population will be uniformly persistent in the sense that
   there
exists  $\xi_0>0$ such that for any $\phi \in \mathbb{X}_+\setminus\{\mathbf{0}\}$, the solution $u(t,x)$ of \eqref{model1dge}-\eqref{udbcs} with the initial data $\phi$ satisfies \begin{equation}\label{uninequality1}
\liminf_{t\rightarrow \infty}u(t,x)\geq
\xi_0,\ \mbox{ uniformly for } x\in[0,L],
\end{equation}
in the case of Robin boundary conditions or
\begin{equation}\label{uninequality2}
\liminf_{t\rightarrow \infty}u(t,x)
\geq\xi_0
 \tilde{e}(x),\  \mbox{ uniformly for } x\in[0,L],
\end{equation}
 in the case of Dirichlet boundary
conditions, where $\tilde{e}(\cdot)$ is a given positive element in
$Int(\mathbb{X}_+)$.
  \item[(ii)] If $\lambda_1(\alpha,l)>0$, then
  there exist small initial data in $\mathbb{X}_+$ such that
   $\limsup\limits_{t\rightarrow \infty} u(t,x)=0$ uniformly for $x\in [0,L]$.
  \end{enumerate}
}
\end{Theorem}
\begin{proof}
(i).
Note that the problem  \eqref{model1dge}-\eqref{udbcs}  has a unique solution in  $\mathbb{X}$ and the solution to
 \eqref{model1dge}-\eqref{udbcs} remains nonnegative on its interval of existence if it is nonnegative initially (see e.g.,
\cite{JinPengShi,von1988}).
Let
$\{\Phi_t\}_{t\geq 0}$ be the semiflow generated by
\eqref{model1dge}-\eqref{udbcs} such that $u(t,x)=\Phi_t(u_0)(x)$, where $u(t,x)$ is the solution of \eqref{model1dge}-\eqref{udbcs} with initial condition $u_0$. Then $\Phi_t:\mathbb{X}_+\rightarrow\mathbb{X}_+$
 is point dissipative since for any
initially nonnegative continuous data $u_0$,
 the solution $u(t,x)$  satisfies $0\leq \limsup\limits_{t\rightarrow \infty}u(t,x)\leq 1$ for all $x\in [0,L]$. Moreover, the maximum principle also implies that for any $v\in \mathbb{X}_+\setminus\{\mathbf{0}\}$, $\Phi_t(v)(x)>0$ for all $t>0$ and $x\in [0,L]$ except at the Dirichlet boundary point(s). Obviously,
$\Phi_t:\mathbb{X}_+\rightarrow\mathbb{X}_+$ is  compact for all $t>0.$
     By \cite[Theorem 3.4.8]{Hale-1988}, it
follows that
$\Phi_t:\mathbb{X}_+\rightarrow\mathbb{X}_+$ ($t>0$), has a global compact attractor.

We now state the following claim: there exists $\delta_0>0$ such that
$$\limsup\limits_{t\rightarrow \infty}||\Phi_t(v)||\geq \delta_0,\ \forall v\in \mathbb{X}_+\setminus\{\mathbf{0}\}.$$

\noindent To this end, let $\varphi_1\in \mathbb{X}_+$
be a positive eigenfunction of the eigenvalue
problem (\ref{model1dgelineig0}) corresponding to the principal
eigenvalue $\lambda_1(\alpha,l)<0$. Then there exists a
sufficiently small $\epsilon_0>0$ such that for any $\epsilon\in (0,\epsilon_0)$, the eigenvalue problem
\begin{equation}\label{model1dgelineig0ep}
\left\{
\begin{array}{ll}
\D -\varphi_{xx}-(f'(0)-\epsilon)\varphi=\lambda\varphi, &\quad  x\in (\alpha,\alpha+l),\\
\D -\varphi_{xx}-(g'(0)-\epsilon)\varphi=\lambda\varphi, & \quad  x\in (0,\alpha)\cup (\alpha+l,L),\\
\alpha_1\varphi'(0)-\alpha_2 \varphi(0)=0,
& \quad
\beta_1\varphi'(L)+\beta_2 \varphi(L)=0
\end{array}
\right.
\end{equation}
admits a principal eigenvalue $\lambda_1^\epsilon<0$ with a
positive eigenfunction $\varphi_1^{\epsilon}\in \mathbb{X}_+$.
Let $\bar{\epsilon}\in (0,\epsilon_0)$.
 By the continuity of $f$ and $g$,
 there exists a $\delta_0>0$ such that $f(u)>(f'(0)-\bar{\epsilon})u$ and $g(u)>(g'(0)-\bar{\epsilon})u$ when
 $0<u<\delta_0$ for all $x\in [0,L]$.
      Assume, for the sake of contradiction, that there exists $v_0\in \mathbb{X}_+\setminus\{\mathbf{0}\}$ such that the solution $u(t,x)=\Phi_t(v_0)(x)$ of \eqref{model1dge}-\eqref{udbcs} with the initial value $v_0$ satisfies
\begin{equation}\label{eqA.100}
\limsup\limits_{t\rightarrow \infty}||\Phi_t(v_0)||<\delta_0.
\end{equation}
Then there
exists a large $t_0>0$, such that $u(t,x)<\delta_0$ for all $x\in[0,L]$ and $t\geq t_0$, and hence,
$f(u)>(f'(0)-\bar{\epsilon})u$ and $g(u)>(g'(0)-\bar{\epsilon})u$ for all $x\in[0,L]$ and $t\geq
t_0.$
Therefore, it holds that
\begin{equation}\label{model1dgeproof}
\left\{
\begin{array}{ll}
\D u_t=u_{xx}+f(u)> u_{xx}+(f'(0)-\bar{\epsilon})u, &\, x\in (\alpha,\alpha+l),\\
\D u_t=u_{xx}+g(u)>u_{xx}+(g'(0)-\bar{\epsilon})u, &\, x\in (0,\alpha)\cup (\alpha+l,L),
\end{array}
\right.
\end{equation}
for all $t\geq t_0$. Since $u(t_0,\cdot)\in \mbox{Int } \mathbb{X}_+$, we can choose a sufficiently small number
$\eta>0$, such that $u(t_0,\cdot)\geq
\eta\varphi_1^{\bar{\epsilon}}(\cdot)$, where
$\varphi_1^{\bar{\epsilon}}$ is the positive eigenfunction of
(\ref{model1dgelineig0ep}) (with $\epsilon=\bar{\epsilon}$) corresponding
to $\lambda_1^{\bar{\epsilon}}$.
 Note that  for
$t\geq t_0$, $\eta
e^{-\lambda_1^{\bar{\epsilon}}(t-t_0)}\varphi_1^{\bar{\epsilon}}(x)$ is the unique
solution of
\begin{equation}\label{model1dgeproof}
\left\{
\begin{array}{l}
\D \bar{u}_t= \bar{u}_{xx}+(f'(0)-\bar{\epsilon})\bar{u}, \quad x\in (\alpha,\alpha+l),\ \ t>t_0,\\
\D \bar{u}_t=\bar{u}_{xx}+(g'(0)-\bar{\epsilon})\bar{u}, \quad x\in (0,\alpha)\cup (\alpha+l,L),\ \ t>t_0,\\
\mbox{ with boundary conditions as in  (\ref{udbcs})},\\
\bar{u}(t_0,x)=
\eta\varphi_1^{\bar{\epsilon}}(x),\ \ \ \ \ \ \,\ \ \ \ x\in(0,L).
\end{array}
\right.
\end{equation}
Then it follows from the comparison principle that
$u(t,x)\geq \eta
e^{-\lambda_1^{\bar{\epsilon}}(t-t_0)}\varphi_1^{\bar{\epsilon}}(x)$
for all $ x\in[0,L], t\geq t_0$,
and hence,
$\textstyle\max_{x\in[0,L]}u(t,x)\rightarrow \infty$ as $t
\rightarrow \infty$, which contradicts (\ref{eqA.100}).
Therefore,
$\limsup\limits_{t\rightarrow \infty}||\Phi_t(v)||\geq \delta_0$ for all $ v\in \mathbb{X}_+\setminus\{\mathbf{0}\}.$
 The claim is proved.

For $\phi\in \mathbb{X}_+$, we define $p(\phi):=\min\limits_{x\in[0,L]}\phi(x)$ in the case of Robin boundary conditions for
\eqref{model1dge}-\eqref{udbcs} or $p(\phi):=\sup\{\beta\in \mathbb{R}_+:
\phi(x)\geq \beta \tilde{e}(x), \forall x\in[0,L]\}$ in the case of Dirichlet boundary
conditions for \eqref{model1dge}-\eqref{udbcs}, where $\tilde{e}$ is a given element in
$Int(C_0^1([0,L],\mathbb{R}_+))$.
Clearly, $p^{-1}(0,\infty)\subseteq \mathbb{X}_+\setminus\{\mathbf{0}\}$. Furthermore, one can show that $p$ is a generalized distance function for the semiflow $\Phi_t$ in the sense that $p$ has the property that if $p(\phi)>0$ or
$\phi\in \mathbb{X}_+\setminus\{\mathbf{0}\}$ with $p(\phi)=0$, then $p(\Phi_t(\phi))>0$ for all $\ t>0$ (see, e.g., \cite{Smithzhao2001}). By the above claim, it follows that  $W^{s}(\{\mathbf{0}\})\cap
(\mathbb{X}_+\setminus\{\mathbf{0}\})=\emptyset$, where $W^{s}(\{\mathbf{0}\})$
is the stable set of $\{\mathbf{0}\}$ (see \cite{Smithzhao2001}).
It then follows from Theorem 3
in \cite{Smithzhao2001} that there exists $\xi_0>0$ such that
$\liminf\limits_{t\rightarrow \infty}p(\Phi_t(\phi))>\xi_0$ for any
$\phi\in \mathbb{X}_+\setminus\{\mathbf{0}\}$. That is, $[\Phi_t(\phi)](x)>\xi_0$
in the case of Robin boundary conditions or $[\Phi_t(\phi)](x)>
\xi_0\tilde{e}(x)$  in the case of Dirichlet boundary
conditions, for any $\phi\in \mathbb{X}_+\setminus\{\mathbf{0}\}$ and sufficiently large $ t$. Hence, (\ref{uninequality1}) and (\ref{uninequality2}) are proved.

(ii).
Note that $u=0$ is locally asymptotically stable for \eqref{model1dge}-\eqref{udbcs} if $\lambda_1(\alpha,l)>0$. In this case, for any sufficiently small initial value $v\in \mathbb{X}_+\setminus\{\mathbf{0}\}$ (i.e., when $||v||$ is sufficiently small), the solution $u(t,x)$ of \eqref{model1dge}-\eqref{udbcs} through $v $ satisfies $\limsup\limits_{t\rightarrow \infty}u(t,x)=0$ and hence $\lim\limits_{t\rightarrow \infty}u(t,x)=0$ uniformly for all $x\in [0,L]$. Therefore, the population dies out in the long term for such initial data.
\end{proof}

In the following,
we give sufficient conditions to guarantee persistence and extinction of the species respectively. Firstly,
 we provide an explicit condition to ensure
 population persistence in terms of model parameters of \eqref{model1dge}-\eqref{udbcs}.
Consider the auxiliary scalar
equation with  Dirichlet boundary conditions on
$\bar{\Omega}_0=[\alpha,\alpha+l]$:
\begin{equation}\label{aux}
\begin{cases}
u_{xx}+f(u)=0,\quad &x\in (\alpha,\alpha+l),\\
u=0,\quad &x=\alpha, \alpha+l.
\end{cases}
\end{equation}
 It is well-known (see
e.g., Theorem 2.5 in \cite{OSS}) that  \eqref{aux} has a unique
positive solution satisfying $0<\theta_{f}(x)<1$ for $x\in(\alpha,\alpha+l)$
when $f'(0)>\lambda_0(\Omega_0)$, where $\lambda_0(\Omega_0)=\pi^2/l^2$ is the unique principal eigenvalue of the eigenvalue problem
\begin{equation}\label{auxeigen}
\begin{cases}
\phi_{xx}=-\lambda \phi,\quad &x\in (\alpha,\alpha+l),\\
\phi=0,\quad &x=\alpha, \alpha+l.
\end{cases}
\end{equation}
Then we can prove the following
result for population persistence in a habitat with a  large protection
zone.

\begin{Proposition}\label{persistence}
{\rm
Suppose that $f'(0)>\pi^2/l^2$, i.e., $l$ is properly large but satisfies $0<l\leq \alpha+l\leq L$.
Then any nonnegative solution $u(t,x)$ ($u\not\equiv 0$) of \eqref{model1dge}-\eqref{udbcs}
satisfies $$\liminf\limits_{t\rightarrow
\infty}u(t,x)\geq \underline{u}(x)$$ for all  $x\in[0,L]$, where
\begin{equation}\label{lower}
\underline{u}(x)=
\begin{cases}
\theta_{f}(x),\quad &x\in (\alpha,\alpha+l),\\
0,\quad &x\in(0,L)\backslash (\alpha,\alpha+l).
\end{cases}
\end{equation}
}
\end{Proposition}
\begin{proof}
 Let $\bar{u}=\max\{1,\max\limits_{x\in[0,L]}u_0(x)\}$. Then $\underline{u}(x)<\bar{u}$ on $[0,L]$.
 Clearly,
$\bar{u}$ is an upper solution  of \eqref{model1dge}. By using the glued local lower solution in $[0,L]$ (see Theorem 1.25 in \cite{CosnerContrall}), we know that $\underline{u}(x)$ defined in \eqref{lower} is a lower solution of \eqref{model1dge}. Then by the upper-lower solution theory, the nonnegative solution $u(t,x)$ of \eqref{model1dge} satisfies
 \begin{equation*}
 \underline{u}(x)\leq \liminf _{t\rightarrow \infty}u(t,x),
  \;\; x\in[0,L]
 \end{equation*}
for any initial value $u_0$ satisfying $u_0\in \mathbb{X}$, $ u_0\geq,\not\equiv0$.
\end{proof}

Next, we  give a sufficient condition for population extinction when one of the boundary conditions is not Neumann.
Let $\lambda_0(L)$ be the principal eigenvalue of the
eigenvalue problem:
\begin{equation}\label{lambda0LBCs}
\begin{cases}
-\varphi''=\lambda\varphi, &\quad x\in(0,L),\\
\alpha_1\varphi'(0)-\alpha_2\varphi(0)=0,
& \quad \beta_1\varphi'(L)+\beta_2\varphi(L)=0.
\end{cases}
\end{equation}
The value of $\lambda_0(L)$ or the conditions that $\lambda_0(L)$ should satisfy
are given in Table \ref{lambda0Ltable}.
\begin{table}[htp]\centering \label{t1}
\begin{tabular}{|c|c|c|c|}
\hline
\backslashbox{$x=0$}{$x=L$} & $\begin{array}{l}
\mbox{ Dirchlet }\\
\varphi(L)=0
\end{array}$ & $\begin{array}{l}
\mbox{  Neumann }\\
\varphi'(L)=0
\end{array}$ & $\begin{array}{l}
\mbox{ Robin  } (\beta_i>0)\\
\beta_1\varphi'(L)+\beta_2\varphi(L)=0
\end{array}$  \\
\hline
$\begin{array}{l}
\mbox{ Dirchlet }\\
\varphi(0)=0
\end{array}$ & $\dfrac{\pi^2}{L^2}$ & $\dfrac{\pi^2}{4L^2}$ & $\tan{\sqrt{\lambda_0}L}=-\dfrac{\beta_1\sqrt{\lambda_0}}{\beta_2}$  \\
\hline
$\begin{array}{l}
\mbox{  Neumann }\\
\varphi'(0)=0
\end{array}$& $\dfrac{\pi^2}{4L^2}$ & $0$ & $\tan{\sqrt{\lambda_0}L}=\dfrac{\beta_2}{\beta_1\sqrt{\lambda_0}}$ \\
\hline
$\begin{array}{l}
\mbox{ Robin  }(\alpha_i>0)\\
\alpha_1\varphi'(0)-\alpha_2\varphi(0)=0
\end{array}$  & $\tan{\sqrt{\lambda_0}L}=-\dfrac{\alpha_1\sqrt{\lambda_0}}{\alpha_2}$ & $\tan{\sqrt{\lambda_0}L}=\dfrac{\alpha_2}{\alpha_1\sqrt{\lambda_0}}$ & $\cot{\sqrt{\lambda_0}L}=\dfrac{\alpha_1\beta_1 \lambda_0-\alpha_2\beta_2}{\sqrt{\lambda_0}(\alpha_2\beta_1+\alpha_1\beta_2)}$ \\
\hline
\end{tabular}
\caption{The values or conditions for the principle eigenvalue $\lambda_0(L)$ of (\ref{lambda0LBCs}) under different boundary conditions. For simplicity, we denote $\lambda_0(L)$ by $\lambda_0$ in the table.}\label{lambda0Ltable}
\end{table}
Let $\lambda_1^\ast:=\lambda_0(L)-f'(0)$ be the principal eigenvalue of
\begin{equation}\label{newproeq1}
\begin{cases}
-\varphi_{xx}-f'(0)\varphi=\lambda \varphi, \quad x\in(0,L),\\
\mbox{the boundary condition at }x=0 \mbox{ or } x=L  \mbox{ is not Neumann.}
\end{cases}
\end{equation}

Before stating our result, we need to introduce a so-called subhomogeneous condition for $f$ and $g$ based on {\bf (A1)} and {\bf (A2)}. They are typically satisfied by the examples in  (\ref{growthfunex}).

\begin{enumerate}
\item[{\bf (A3)}] $f(u)/u\leq f'(0)$ and $g(u)/u\leq -g'(0)$ for any $u>0$, and $f'(0)\geq -g'(0)$.
\end{enumerate}

\begin{Proposition}\label{sufficonex}
{\rm
Assume that {\bf (A3)} holds.
Let $u(t,x)$ be the solution of \eqref{model1dge}-\eqref{udbcs} with the same boundary conditions as in (\ref{newproeq1}) for $t>0$.
 If $\lambda_1^\ast>0$, then $\lim\limits_{t\rightarrow \infty}u(t,x)=0$ uniformly for all $x\in [0,L]$.
}
\end{Proposition}

\begin{proof}
Consider
\begin{equation}\label{2-1}
\begin{cases}
\bar{u}_t=\bar{u}_{xx}+f'(0)\bar{u}, \quad x\in (\alpha,\alpha+l),\,t>0,\\
\bar{u}_t=\bar{u}_{xx}-g'(0)\bar{u}, \quad x\in (0,\alpha)\cup (\alpha+l,L),\,t>0,\\
\bar{u}(0,x)=u_0(x)\geq 0, \quad x\in (0,L),\\
\bar{u}(t,0) \mbox{ and } \bar{u}(t,L) \mbox{ satisfy the same boundary conditions as in (\ref{newproeq1}) for } t>0.
\end{cases}
\end{equation}
From (A1)-(A3), $f(\bar{u})\leq f'(0)\bar{u}$ and $g(\bar{u})\leq -g'(0)\bar{u}\leq f'(0)\bar{u}$ for all $\bar{u}\geq 0$.
By the comparison principle (see e.g., \cite[Lemma A.3]{JinPengShi}), all nonnegative solutions of \eqref{model1dge} satisfy $u(t,x)\leq \bar{u}(t,x)\leq\tilde{u}(t,x)$, given the same initial value condition $u_0$, where $\tilde{u}(t,x)$ is the solution of
\begin{equation}\label{2-2}
\begin{cases}
\tilde{u}_t=\tilde{u}_{xx}+f'(0)\tilde{u}, \quad x\in (0,L),\,t>0,\\
\tilde{u}(0,x)=u_0(x)\geq 0, \quad x\in (0,L),\\
\tilde{u}(t,0) \mbox{ and } \tilde{u}(t,L) \mbox{ satisfy the same boundary conditions as in (\ref{newproeq1}) for } t>0.
\end{cases}
\end{equation}
 Thus, $\lim\limits_{t\rightarrow \infty}\tilde{u}(t,x)=0$ uniformly for all $x\in [0,L]$ if $\lambda_1^\ast>0$, which completes the proof.
\end{proof}

\begin{Remark}
{\rm
Proposition \ref{persistence} implies that the population will not be
extinct when the length of the protection zone $l$ is larger than the
threshold $\pi/\sqrt{f'(0)}$. In this case, no matter how small the
initial value is, provided that it is not zero, the population will not be wiped out at least in the protection zone. Note that $\lambda_1(\alpha,l)$ is not used in the proof of Proposition \ref{persistence}.  We will confirm in the next section that $\lambda_1(\alpha,l)<0$ follows from the assumption $f'(0)>\pi^2/l^2$ for \eqref{model1dge}-\eqref{udbcs} with some specific boundary conditions. Since $\lambda_0(L)$ decreases in $L$, we know that $\lambda_0(L)>f'(0)$ for small $L$. Thus, Proposition \ref{sufficonex} indicates that the population will be extinct if the total habitat length $L$ is smaller than some critical value, except in the case where the Neumann boundary condition is applied at $x=0,L$ and the population will always be extinct in this case despite the initial value distribution.
}
\end{Remark}

In the rest of the paper, we say that the population will {\it persist } if $\lambda_1(\alpha,l)<0$ and the population will be {\it extinct } if $\lambda_1(\alpha,l)>0$. In particular, since our purpose is to protect endangered species with a small initial population density, by saying that the population will be extinct, we mean that the population will die out if its initial distribution is sufficiently small.

\section{The effect of the protection zone on population
persistence}

 Proposition \ref{persistence} indicates that a large size of the protection zone guarantees population persistence at least in the protection zone.
In this section, we investigate the dependence of $\lambda_1(\alpha,l)$ on $\alpha$ and $l$ and then provide more details about how the protection zone  $[\alpha,\alpha+l]$ influences population persistence or extinction. In this section, we always assume {\bf (A1)}, {\bf (A2)} and  $f'(0)\geq -g'(0)$ in {\bf (A3)}.

\subsection{The relation between $\lambda_1(\alpha,l)$ and $l$}

Note that $\lambda_1(\alpha,l)$ is a continuous and differentiable function in $\alpha\in [0,L-l]$ or in $l\in [0,L-\alpha]$; see e.g., \cite{Hessbook}. By a simple observation, we have the following result.
\begin{Lemma}\label{lambdadeplth}
{\rm
For any fixed $\alpha\in (0,L)$, $\lambda_1(\alpha,l)$ decreases
with respect to $l\in(0,L-\alpha)$.
}
\end{Lemma}
\begin{proof}
Given an $\alpha\in (0,L)$, $H$ defined in \eqref{model1dgelineig0} is decreasing
with regard to $l$ since $-f'(0)<0<-g'(0)$. By the monotonicity of
the principal eigenvalue $\lambda_1(\alpha,l)$ of \eqref{model1dgelineig0} in
the reaction function $H$ (see e.g., \cite[Lemma 15.5]{Hessbook}), we know that $\lambda_1(\alpha,l)$ also
decreases in $l$.
\end{proof}

\begin{Remark}
{\rm
Lemma \ref{lambdadeplth} implies that the longer the protection zone is, the easier it is for the population to persist in the whole habitat.
}
\end{Remark}

\subsection{The relation between $\lambda_1(\alpha,l)$ and
$\alpha$}\label{lambdaalphader}

In the following, we investigate the dependence of $\lambda_1(\alpha,l)$ on $\alpha$.
Note that $\lambda_1(\alpha,l)$ and $\varphi_1$ satisfy
\begin{equation}\label{model1dgelineig}
\left\{
\begin{array}{ll}
\D -\varphi_{1,xx}-f'(0)\varphi_1=\lambda_1(\alpha,l)\varphi_1, &\quad  x\in (\alpha,\alpha+l),\\
\D -\varphi_{1,xx}-g'(0)\varphi_1=\lambda_1(\alpha,l)\varphi_1, &\quad  x\in (0,\alpha)\cup (\alpha+l,L),\\
\alpha_1\varphi_1'(0)-\alpha_2 \varphi_1(0)=0,
&\quad
\beta_1\varphi_1'(L)+\beta_2 \varphi_1(L)=0.
\end{array}
\right.
\end{equation}

We can derive some estimates for $\lambda_1(\alpha,l)$ as follows.
 Consider the eigenvalue problems
\begin{equation}\label{feigen}
\begin{cases}
-\varphi_{xx}-f'(0)\varphi=\lambda \varphi, &\quad x\in (0,L),\\
\alpha_1\varphi'(0)-\alpha_2 \varphi(0)=0,
&\quad
\beta_1\varphi'(L)+\beta_2 \varphi(L)=0,
\end{cases}
\end{equation}
and
\begin{equation}\label{geigen}
\begin{cases}
-\varphi_{xx}-g'(0)\varphi=\lambda \varphi,&\quad x\in (0,L),\\
\alpha_1\varphi'(0)-\alpha_2 \varphi(0)=0,
&\quad
\beta_1\varphi'(L)+\beta_2 \varphi(L)=0.
\end{cases}
\end{equation}
Denote by $\lambda_{L,f}$ and $\lambda_{L,g}$  the principle
eigenvalues of (\ref{feigen}) and (\ref{geigen}), respectively. The following result is valid.

\begin{Lemma}\label{lambdafg1}
{\rm
Let $\lambda_1(\alpha,l)$ be the principal eigenvalue of \eqref{model1dgelineig0}. Then
\begin{equation*}
\lambda_{L,f}<\lambda_1(\alpha,l)<\lambda_{L,g}.
\end{equation*}
}
\end{Lemma}

\begin{proof}
Note that $-f'(0)\not\equiv, \leq H\leq, \not\equiv -g'(0)$ on $ [0,L]$.
The result follows from the monotonicity of
the principal eigenvalue $\lambda_1(\alpha,l)$ of \eqref{model1dgelineig0} in
the reaction function $H$ (see e.g., \cite[Lemma 15.5]{Hessbook}).
\end{proof}

We first show that $f'(0)+\lambda_1(\alpha,l)>0$ is always true.
\begin{Lemma}\label{lambda1fp0}
{\rm
$\lambda_1(\alpha,l)+f'(0)>0$ is always true  under the general boundary conditions in \eqref{udbcs}.
}
\end{Lemma}
\begin{proof}
We know that $\lambda_{L,f}=\lambda_0(L)-f'(0)$ and $\lambda_{L,g}=\lambda_0(L)-g'(0)$.
Lemma \ref{lambdafg1} implies
\begin{equation}\label{lambdacomparison}
\lambda_0(L)-f'(0)=\lambda_{L,f}<\lambda_1(\alpha,l)<\lambda_{L,g}=\lambda_0(L)-g'(0).
\end{equation}
This  indicates that $\lambda_1(\alpha,l)+f'(0)>0$ if $\lambda_0(L)\geq 0$.
Therefore, by Table \ref{lambda0Ltable}, $\lambda_1(\alpha,l)+f'(0)>0$ is always true for boundary conditions  in (\ref{udbcs}).
\end{proof}

In the following, we derive the relation between $\lambda_1(\alpha,l)$ and $\alpha$ when $g'(0)+\lambda_1(\alpha,l)$ has different signs.
Denote
\begin{equation}\label{titldefandg}
\tilde{f}:=\sqrt{f'(0)+\lambda_1(\alpha,l)},\quad \tilde{g}:=\sqrt{|g'(0)+\lambda_1(\alpha,l)|}.
\end{equation}

We first consider the case of

\noindent {\bf (H1).} $g'(0)+\lambda_1(\alpha,l)<0.
$\\
In this case, a simple analysis shows that $\varphi_1$ on $[0,L]$
can be expressed as
\begin{equation}\label{varphi1solution}
\varphi_1(x)=\left\{
\begin{array}{ll}
C_3e^{-\tilde{g}x}+C_4e^{\tilde{g}x}, &
\text{for}\;\;x\in[0,\alpha],\\
C_1\cos\tilde{f}x+C_2\sin\tilde{f}x, &
\text{for}\;\;x\in[\alpha,\alpha+l],\\
C_5e^{-\tilde{g}x}+C_6e^{\tilde{g}x}, &
\text{for}\;\;x\in[\alpha+l,L],
\end{array}
\right.
\end{equation}
where each $C_i$ ($i=1,\cdots,6$) is a constant.
By the boundary conditions of $\varphi_1$ at $0$ and $L$, we have
$$
\begin{array}{l}
\alpha_1\varphi_1'(0)-\alpha_2 \varphi_1(0)=\alpha_1 (-C_3 {\tilde{g}}+C_4 {\tilde{g}})-\alpha_2(C_3+C_4)=0,\\
\beta_1\varphi_1'(L)+\beta_2 \varphi_1(L)=\beta_1(-C_5 {\tilde{g}} e^{-{\tilde{g}}L}+C_6 {\tilde{g}} e^{{\tilde{g}}L})+\beta_2(C_5 e^{-{\tilde{g}}L}+C_6 e^{{\tilde{g}}L})=0,
\end{array}
$$
which imply $C_3=R_1 C_4$ and $C_6=R_2 C_5$ with
 \begin{equation}\label{R12def}
R_1=\frac{\alpha_1 {\tilde{g}}-\alpha_2}{\alpha_1 {\tilde{g}}+\alpha_2},\quad R_2=e^{-2{\tilde{g}}L}\frac{\beta_1 {\tilde{g}}-\beta_2}{\beta_1 {\tilde{g}}+\beta_2}.
\end{equation}
Substituting $C_3$ and $C_6$ into (\ref{varphi1solution}) yields
\begin{equation*}
\varphi_1(x)=\left\{
\begin{array}{ll}
R_1 C_4e^{-\tilde{g}x}+C_4e^{\tilde{g}x}, &
\text{for}\;\;x\in[0,\alpha],\\
C_1\cos\tilde{f}x+C_2\sin\tilde{f}x, &
\text{for}\;\;x\in[\alpha,\alpha+l],\\
C_5e^{-\tilde{g}x}+R_2 C_5e^{\tilde{g}x}, &
\text{for}\;\;x\in[\alpha+l,L],
\end{array}
\right.
\end{equation*}
and
\begin{equation*}
\varphi_1'(x)=\left\{
\begin{array}{ll}
-R_1 C_4 {\tilde{g}}e^{-\tilde{g}x}+C_4 {\tilde{g}} e^{\tilde{g}x}, &
\text{for}\;\;x\in[0,\alpha],\\
-C_1 {\tilde{f}} \sin\tilde{f}x+C_2 {\tilde{f}} \cos\tilde{f}x, &
\text{for}\;\;x\in[\alpha,\alpha+l],\\
-C_5 {\tilde{g}}e^{-\tilde{g}x}+R_2 C_5 {\tilde{g}} e^{\tilde{g}x}, &
\text{for}\;\;x\in[\alpha+l,L].
\end{array}
\right.
\end{equation*}
In view of the fact that $\varphi_1$ is continuously differentiable at
$\alpha$ and $\alpha+l$, we find that
\begin{equation}\label{RR1}
R_1C_4e^{-{\tilde{g}}\alpha}+C_4e^{{\tilde{g}}\alpha}=C_1\cos\tilde{f}\alpha+C_2\sin\tilde{f}\alpha,
\end{equation}
\begin{equation}\label{RR3}
-C_4 R_1 {\tilde{g}}e^{-\tilde{g}\alpha}+C_4 {\tilde{g}}e^{\tilde{g}\alpha}=-C_1 {\tilde{f}}\sin\tilde{f}\alpha+C_2 {\tilde{f}} \cos\tilde{f}\alpha,
\end{equation}
\begin{equation}\label{RR2}
C_5e^{-\tilde{g}(\alpha+l)}+R_2 C_5e^{\tilde{g}(\alpha+l)}=C_1\cos\tilde{f}(\alpha+l)+C_2\sin\tilde{f}(\alpha+l),
\end{equation}
\begin{equation}\label{RR4}
-C_5\tilde{g}e^{-\tilde{g}(\alpha+l)}+C_5 R_2 {\tilde{g}}e^{\tilde{g}(\alpha+l)}
=-C_1 {\tilde{f}}\sin\tilde{f}(\alpha+l)+C_2 {\tilde{f}}\cos\tilde{f}(\alpha+l).
\end{equation}
Then
$\eqref{RR1}\times\eqref{RR2}+\dfrac{\eqref{RR3}\times\eqref{RR4}}{\tilde{f}^2}$
leads to
\begin{equation}\label{RR01}
\begin{array}{rl}
&C_4 C_5(R_1e^{-{\tilde{g}}\alpha}+e^{{\tilde{g}}\alpha})(e^{-\tilde{g}(\alpha+l)}+R_2 e^{\tilde{g}(\alpha+l)})+C_4 C_5\frac{(- R_1 {\tilde{g}}e^{-\tilde{g}\alpha}+ {\tilde{g}}e^{\tilde{g}\alpha})(-\tilde{g}e^{-\tilde{g}(\alpha+l)}+ R_2 {\tilde{g}}e^{\tilde{g}(\alpha+l)})}{{\tilde{f}}^2}\\
=&(C_1\cos\tilde{f}\alpha+C_2\sin\tilde{f}\alpha)(C_1\cos\tilde{f}(\alpha+l)+C_2\sin\tilde{f}(\alpha+l))\\
&+(-C_1 \sin\tilde{f}\alpha+C_2  \cos\tilde{f}\alpha)(-C_1 \sin\tilde{f}(\alpha+l)+C_2 \cos\tilde{f}(\alpha+l))\\
=& (C_1^2+C_2^2)\cos {\tilde{f}}l
\end{array}
\end{equation}
and $\eqref{RR1}\times\eqref{RR4}-\eqref{RR3}\times\eqref{RR2}$ leads to
\begin{equation}\label{RR02}
\begin{array}{rl}
&C_4 C_5(R_1e^{-{\tilde{g}}\alpha}+e^{{\tilde{g}}\alpha})(-\tilde{g}e^{-\tilde{g}(\alpha+l)}+ R_2 {\tilde{g}}e^{\tilde{g}(\alpha+l)})-C_4 C_5(- R_1 {\tilde{g}}e^{-\tilde{g}\alpha}+ {\tilde{g}}e^{\tilde{g}\alpha})(e^{-\tilde{g}(\alpha+l)}+R_2 e^{\tilde{g}(\alpha+l)})\\
=&(C_1\cos\tilde{f}\alpha+C_2\sin\tilde{f}\alpha)(-C_1 {\tilde{f}}\sin\tilde{f}(\alpha+l)+C_2 {\tilde{f}}\cos\tilde{f}(\alpha+l))\\
& -(-C_1 {\tilde{f}}\sin\tilde{f}\alpha+C_2 {\tilde{f}} \cos\tilde{f}\alpha)(C_1\cos\tilde{f}(\alpha+l)+C_2\sin\tilde{f}(\alpha+l))\\
=& -(C_1^2+C_2^2){\tilde{f}} \sin {\tilde{f}}l.
\end{array}
\end{equation}
By simplifying (\ref{RR01}) and (\ref{RR02}), we obtain
that $\lambda_1(\alpha,l)$ satisfies
\begin{equation}\label{RRtanequ}
\begin{array}{rl}
\tan
\tilde{f}l&=\dfrac{\tilde{f}\tilde{g}\left(\left(e^{-\tilde{g}(\alpha+l)}+R_2e^{\tilde{g}(\alpha+l)}\right)(-R_1e^{-\tilde{g}\alpha}+e^{\tilde{g}\alpha})
-(R_1e^{-\tilde{g}\alpha}+e^{\tilde{g}\alpha})\left(-e^{-\tilde{g}(\alpha+l)}+
R_2e^{\tilde{g}(\alpha+l)}\right)\right)}
{\tilde{f}^2(R_1e^{-\tilde{g}\alpha}+e^{\tilde{g}\alpha})\left(e^{-\tilde{g}(\alpha+l)}+R_2e^{\tilde{g}(\alpha+l)}\right)
+\tilde{g}^2(-R_1e^{-\tilde{g}\alpha}+e^{\tilde{g}\alpha})\left(-e^{-\tilde{g}(\alpha+l)}+
R_2e^{\tilde{g}(\alpha+l)}\right)}.
\end{array}
\end{equation}
For simplicity, we introduce the following notations:
\begin{equation}\label{ABCDT1T2def}
\begin{array}{ll}
A=R_1e^{-\tilde{g}\alpha}+e^{\tilde{g}\alpha}, & B=e^{-\tilde{g}(\alpha+l)}+R_2e^{\tilde{g}(\alpha+l)},\\
C=-R_1e^{-\tilde{g}\alpha}+e^{\tilde{g}\alpha}, & D= -e^{-\tilde{g}(\alpha+l)}+
R_2e^{\tilde{g}(\alpha+l)},\\
T_1=BC-AD, & T_2=\tilde{f}^2 AB+\tilde{g}^2 CD,\quad  T_3=BC+AD,\quad T_4=AB{\tilde{f}}^2-CD{\tilde{g}}^2,\\
\hat{R_1}=\D \frac{2\alpha_1 \alpha_2}{(\alpha_1 {\tilde{g}}+\alpha_2)^2},& \hat{R_2}=\D \frac{2\beta_1 \beta_2}{(\beta_1 {\tilde{g}}+\beta_2)^2}.
\end{array}
\end{equation}
Then  (\ref{RRtanequ}) becomes
\begin{equation}\label{RRtanequ1}
\tan {\tilde{f}}l=\frac{{\tilde{f}}{\tilde{g}}(BC-AD)}{{\tilde{f}}^2
AB+{\tilde{g}}^2 CD}=\frac{{\tilde{f}}{\tilde{g}} T_1}{T_2}.
\end{equation}

Now we consider the monotonicity of $\lambda_1(\alpha,l)$ with respect to $\alpha$.
Without causing confusion, for simplicity we will  write $\lambda_1(\alpha,l)$ just as $\lambda_1$ below. Note that
\begin{equation}\label{ABCDdiff}
\begin{array}{ll}
A_\alpha=C({\tilde{g}}_\alpha \alpha+{\tilde{g}})+R_{1,\alpha}e^{-{\tilde{g}}\alpha},&
B_\alpha=D({\tilde{g}}_\alpha(\alpha+l)+{\tilde{g}})+R_{2,\alpha}e^{{\tilde{g}}(\alpha+l)},\\
C_\alpha=A({\tilde{g}}_\alpha \alpha+{\tilde{g}})-R_{1,\alpha}e^{-{\tilde{g}}\alpha},&
D_\alpha=B({\tilde{g}}_\alpha(\alpha+l)+{\tilde{g}})+R_{2,\alpha}e^{{\tilde{g}}(\alpha+l)},\\
R_{1,\alpha}=\hat{R_1}{\tilde{g}}_\alpha,&
R_{2,\alpha}=(-2LR_2+e^{-2{\tilde{g}}L}\hat{R_2}){\tilde{g}}_\alpha,\\
{\tilde{f}}_\alpha={\tilde{f}}_{\lambda_1} \cdot
{\lambda_1}_\alpha=\D \frac{{\lambda_1}_\alpha}{2{\tilde{f}}}, &
{\tilde{g}}_\alpha={\tilde{g}}_{\lambda_1} \cdot
{\lambda_1}_\alpha=\D -\frac{{\lambda_1}_\alpha}{2{\tilde{g}}},
\end{array}
\end{equation}
where subscripts $_\alpha$ and $_{\lambda_1}$ represent partial derivatives with respect to $\alpha$ and $\lambda_1$, respectively.
Differentiating (\ref{RRtanequ1}) with respect $\alpha$ on two sides yields
$$
\begin{array}{l}
\sec^2 {\tilde{f}}l \cdot l \cdot
{\tilde{f}}_\alpha\\
=\D \frac{(({\tilde{g}}_\alpha {\tilde{f}}
+{\tilde{f}}_\alpha {\tilde{g}})T_1+{\tilde{g}}{\tilde{f}}
T_{1,\alpha})T_2-(2{\tilde{f}} {\tilde{f}}_\alpha AB+2{\tilde{g}}
{\tilde{g}}_\alpha CD +{\tilde{f}}^2
A_\alpha B+{\tilde{f}}^2
A B_\alpha+{\tilde{g}}^2C_\alpha D+{\tilde{g}}^2CD_\alpha ) {\tilde{g}}{\tilde{f}}T_1}{T_2^2},
\end{array}
$$
which implies
$$
\begin{array}{rcl}
& &\sec^2 {\tilde{f}}l \cdot l  T_2^2 \frac{{\lambda_1}_\alpha}{2{\tilde{f}}} \\
&=&({\tilde{g}}_\alpha {\tilde{f}} +{\tilde{f}}_\alpha {\tilde{g}})T_1T_2+{\tilde{f}}{\tilde{g}} T_2(
(CD-AB){\tilde{g}}_\alpha l-(B+D)e^{-{\tilde{g}}\alpha}R_{1,\alpha}+(C-A)e^{{\tilde{g}}(\alpha+l)}R_{2,\alpha}
)\\
& &-2{\tilde{f}}^2{\tilde{g}} {\tilde{f}}_\alpha AB T_1
-2{\tilde{f}}{\tilde{g}}^2 {\tilde{g}}_\alpha CD T_1\\
&&
-{\tilde{f}}^3{\tilde{g}}T_1 B(R_{1,\alpha}e^{-{\tilde{g}}\alpha}+C({\tilde{g}}_\alpha \alpha+{\tilde{g}}))
-{\tilde{f}}^3 {\tilde{g}} T_1 A (R_{2,\alpha}e^{{\tilde{g}}(\alpha+l)}+D({\tilde{g}}_\alpha(\alpha+l)+{\tilde{g}}))
\\
&&
 -{\tilde{f}}{\tilde{g}}^3T_1 D(A({\tilde{g}}_\alpha \alpha+{\tilde{g}})-R_{1,\alpha}e^{-{\tilde{g}}\alpha})
 -{\tilde{f}}{\tilde{g}}^3 T_1 C(B({\tilde{g}}_\alpha(\alpha+l)+{\tilde{g}})+R_{2,\alpha}e^{{\tilde{g}}(\alpha+l)})
 \\
&=& (AB{\tilde{f}}^2-CD{\tilde{g}}^2)({\tilde{g}}_\alpha {\tilde{f}}-{\tilde{f}}_\alpha {\tilde{g}})T_1-T_1 {\tilde{g}}_\alpha \alpha(BC+AD){\tilde{f}}{\tilde{g}}({\tilde{f}}^2+{\tilde{g}}^2)\\
&&-T_1 {\tilde{g}}(BC+AD){\tilde{f}}{\tilde{g}}({\tilde{f}}^2+{\tilde{g}}^2)
-{\tilde{f}}{\tilde{g}}(B-D)(B+D)(A^2{\tilde{f}}^2+C^2{\tilde{g}}^2)l {\tilde{g}}_\alpha\\
&& -{\tilde{f}}{\tilde{g}}(A+C)(B^2{\tilde{f}}^2+D^2{\tilde{g}}^2)R_{1,\alpha}e^{-{\tilde{g}}\alpha}-{\tilde{f}}{\tilde{g}}(B-D)(A^2{\tilde{f}}^2+C^2{\tilde{g}}^2)R_{2,\alpha}e^{{\tilde{g}}(\alpha+l)}\\
&=&(AB{\tilde{f}}^2-CD{\tilde{g}}^2)(-\frac{{\tilde{f}}}{2{\tilde{g}}}-\frac{{\tilde{g}}}{2{\tilde{f}}})
T_1{\lambda_1}_\alpha\\
&&
+T_1 \frac{1}{2{\tilde{g}}}{\lambda_1}_\alpha \alpha(BC+AD){\tilde{f}}{\tilde{g}}({\tilde{f}}^2+{\tilde{g}}^2)-T_1 (BC+AD){\tilde{f}}{\tilde{g}}^2({\tilde{f}}^2+{\tilde{g}}^2)\\&&
-{\tilde{f}}{\tilde{g}}(B-D)(B+D)(A^2{\tilde{f}}^2+C^2{\tilde{g}}^2)l (-\frac{{\lambda_1}_\alpha}{2{\tilde{g}}}) -{\tilde{f}}{\tilde{g}}(A+C)(B^2{\tilde{f}}^2+D^2{\tilde{g}}^2)\hat{R_1}(-\frac{{\lambda_1}_\alpha}{2{\tilde{g}}})e^{-{\tilde{g}}\alpha}\\
&&
-{\tilde{f}}{\tilde{g}}(B-D)(A^2{\tilde{f}}^2+C^2{\tilde{g}}^2)(-2LR_2+e^{-2{\tilde{g}}L}\hat{R_2})(-\frac{{\lambda_1}_\alpha}{2{\tilde{g}}})e^{{\tilde{g}}(\alpha+l)}\\
&=&{\lambda_1}_\alpha
[-(\frac{{\tilde{f}}}{2{\tilde{g}}}+\frac{{\tilde{g}}}{2{\tilde{f}}})(AB{\tilde{f}}^2-CD{\tilde{g}}^2)
T_1
+\frac{{\tilde{f}}({\tilde{f}}^2+{\tilde{g}}^2)}{2}T_1  \alpha(BC+AD)
+2{\tilde{f}}(A^2{\tilde{f}}^2+C^2{\tilde{g}}^2)(l-L) R_2 \\&& +{\tilde{f}}(B^2{\tilde{f}}^2+D^2{\tilde{g}}^2)\hat{R_1}
+{\tilde{f}}(A^2{\tilde{f}}^2+C^2{\tilde{g}}^2)e^{-2{\tilde{g}}L}\hat{R_2}]
-{\tilde{f}}{\tilde{g}}^2({\tilde{f}}^2+{\tilde{g}}^2)T_1 (BC+AD).
\end{array}
$$
By (\ref{RRtanequ1}), we also have
$$\sec^2
{\tilde{f}l} T_2^2=(1+\tan^2{\tilde{f}l})T_2^2=T_2^2+{\tilde{f}}^2{\tilde{g}}^2 T_1^2.
$$
Then we finally obtain
\begin{equation}\label{lambdaalphagl0}
\begin{array}{l}
 {\lambda_1}_\alpha=
\D\frac{{\tilde{f}}{\tilde{g}}^2({\tilde{f}}^2
+{\tilde{g}}^2)T_1
T_3}{E_{{\lambda_1}_\alpha}}
\end{array}
\end{equation}
with $$
\begin{array}{rcl}
E_{{\lambda_1}_\alpha}&=&-(\frac{{\tilde{f}}}{2{\tilde{g}}}
+\frac{{\tilde{g}}}{2{\tilde{f}}})T_1T_4+ \frac{{\tilde{f}}({\tilde{f}}^2+{\tilde{g}}^2)}{2} T_1\alpha T_3
-2{\tilde{f}} (A^2{\tilde{f}}^2+C^2{\tilde{g}}^2)(L-l)R_2 +{\tilde{f}}(B^2{\tilde{f}}^2+D^2{\tilde{g}}^2)\hat{R_1}\\
&&
+{\tilde{f}}(A^2{\tilde{f}}^2+C^2{\tilde{g}}^2)e^{-2{\tilde{g}}L}\hat{R_2}
-\frac{ l(T_2^2+{\tilde{f}}^2{\tilde{g}}^2 T_1^2)}{2{\tilde{f}}}.
\end{array}$$
We conclude the above analysis in the following result.
\begin{Lemma}\label{lambdaalphataneqnsgl0}
{\rm
If $g'(0)+\lambda_1(\alpha,l)<0$, then the principal eigenvalue $\lambda_1(\alpha,l)$ of (\ref{model1dgelineig}) satisfies (\ref{RRtanequ}) and (\ref{lambdaalphagl0}), where the related parameters (or notations) are defined in (\ref{titldefandg}), (\ref{R12def}) and (\ref{ABCDT1T2def}).
}
\end{Lemma}

Next we consider the case of \\
\noindent {\bf (H2).} $g'(0)+\lambda_1(\alpha,l)=0.
$\\
 By following a similar process as in the case of (H1), we can obtain the following results; detailed calculations are included in Appendix \ref{applemmalambdaalphataneqnsgeq0}.

 \begin{Lemma}\label{lambdaalphataneqnsgeq0}
{\rm
If $g'(0)+\lambda_1(\alpha,l)=0$, then the principal eigenvalue $\lambda_1(\alpha,l)$ of (\ref{model1dgelineig}) satisfies
the following relations.
\begin{enumerate}
   \item[(i)] When Dirichlet boundary conditions are applied at $x=0,L$, we have    \begin{equation}\label{DDtanequg0}
\begin{array}{rl}
\tan\tilde{f}l
&=\dfrac{\tilde{f}(L-l)}{{\tilde{f}}^2\alpha(L-\alpha-l)-1},
\end{array}
\end{equation}
and
\begin{equation}\label{DDlambdaalphag0}
{\lambda_1}_\alpha=\frac{{\tilde{f}}(L-l)(-{\tilde{f}}^2(L-\alpha-l)+{\tilde{f}}^2\alpha)}{({\tilde{f}}^2\alpha(L-\alpha-l)-1)^2\frac{\sec^2 {\tilde{f}}l\cdot l}{2{\tilde{f}}}-\frac{1}{2{\tilde{f}}}(L-l)({\tilde{f}}^2\alpha(L-\alpha-l)-1)+(L-l){\tilde{f}}\alpha(L-\alpha-l)}.
\end{equation}

   \item[(ii)]  When Neumann boundary condition is applied at $x=0$ and Dirichlet boundary condition is applied at $x=L$, we have
       \begin{equation}\label{NDtanequg0}
\begin{array}{rl}
\tan\tilde{f}l
&=\dfrac{1}{\tilde{f}(L-\alpha-l)},
\end{array}
\end{equation}
and
\begin{equation}\label{NDlambdaalphag0}
{\lambda_1}_\alpha
=\frac{2\tilde{f}^2}{\tilde{f}^2 sec^2\tilde{f}l\cdot l\cdot (L-\alpha-l)^2+L-\alpha-l}.
\end{equation}
 \end{enumerate}
}
\end{Lemma}

Finally, we consider the case of\\
\noindent {\bf (H3).} $g'(0)+\lambda_1(\alpha,l)>0.
$\\
 Similarly as before, we can obtain the following results; detailed calculations are included in Appendix \ref{applemmalambdaalphataneqnsgg0}.

 \begin{Lemma}\label{lambdaalphataneqnsgg0}
{\rm
If $g'(0)+\lambda_1(\alpha,l)>0$, then the principal eigenvalue $\lambda_1(\alpha,l)$ of (\ref{model1dgelineig}) satisfies the following relations.
\begin{enumerate}
   \item[(i)] When Dirichlet boundary conditions are applied at $x=0,L$,  we have
 \begin{equation}\label{DDtanequgg0}
\begin{array}{rl}
\tan\tilde{f}l
&=\dfrac{\tilde{f}\tilde{g}\sin \tilde{g}(L-l)}{\bar{T}_{dd}},
\end{array}
\end{equation}
and
\begin{equation}\label{DDlambdaalphagg0}
{\lambda_1}_\alpha=\frac{({\tilde{f}}{\tilde{g}}^4-{\tilde{f}}^3{\tilde{g}}^2)\sin {\tilde{g}}(L-l)\sin {\tilde{g}}(L-2\alpha-l)}{E_{dd}},
\end{equation}
where $$
\begin{array}{rcl}
\bar{T}_{dd}&=&{\tilde{f}}^2\sin \tilde{g}\alpha\cdot \sin \tilde{g}(L-\alpha-l)-{\tilde{g}}^2 \cos \tilde{g}\alpha \cdot\cos \tilde{g}(L-\alpha-l),\\
E_{dd}&=&\frac{\bar{T}_{dd}^2 \sec^2 {\tilde{f}}l\cdot l}{2{\tilde{f}}}-\frac{1}{2}{\tilde{f}}{\tilde{g}}\sin {\tilde{g}}(L-l)\cos {\tilde{g}}(L-l)-\frac{1}{2{\tilde{g}}}{\tilde{f}}^3\sin {\tilde{g}}(L-l)\sin {\tilde{g}}\alpha \sin {\tilde{g}}(L-\alpha-l)\\
&&+\frac{1}{2{\tilde{f}}}{\tilde{g}}^3\sin {\tilde{g}}(L-l)\cos {\tilde{g}}\alpha \cos {\tilde{g}}(L-\alpha-l)+\frac{1}{2}{\tilde{f}}\tilde{g}^2(L-l)\cos^2 {\tilde{g}}\alpha+\frac{1}{2}{\tilde{f}}^2{\tilde{g}}(L-l)\sin^2{\tilde{g}}\alpha\\
&&+\frac{1}{2}({\tilde{f}}^3-\tilde{f}{\tilde{g}}^2)\alpha \sin {\tilde{g}}(L-l) \sin {\tilde{g}}(L-2\alpha-l).
\end{array}
$$

   \item[(ii)]  When Neumann boundary condition is applied at $x=0$ and Dirichlet boundary condition is applied at $x=L$,   we have
      \begin{equation}\label{NDtanequgg0}
\begin{array}{rl}
\tan\tilde{f}l
&=\dfrac{\tilde{f}\tilde{g}\cos \tilde{g}(L-l)}{\bar{T}_{nd}},
\end{array}
\end{equation}
and
\begin{equation}\label{NDlambdaalphagg0}
{\lambda_1}_\alpha=\frac{\cos {\tilde{g}}(L-l)\cos {\tilde{g}}(L-2\alpha-l)({\tilde{f}}^3{\tilde{g}}^2-{\tilde{f}}{\tilde{g}}^4)}{E_{nd}},
\end{equation}
where $$
\begin{array}{rcl}
\bar{T}_{nd}&=&{\tilde{f}}^2\cos \tilde{g}\alpha\cdot \sin \tilde{g}(L-\alpha-l)+{\tilde{g}}^2 \sin \tilde{g}\alpha \cdot\cos \tilde{g}(L-\alpha-l),\\
E_{nd}&=&\frac{\bar{T}_{nd}^2 \sec^2 {\tilde{f}}l\cdot l}{2{\tilde{f}}}+\frac{1}{2}{\tilde{f}}{\tilde{g}}\cos {\tilde{g}}(L-l)\sin {\tilde{g}}(L-l)-\frac{1}{2{\tilde{g}}}{\tilde{f}}^3\cos {\tilde{g}}(L-l)\cos {\tilde{g}}\alpha \sin {\tilde{g}}(L-\alpha-l)\\
&&-\frac{1}{2{\tilde{f}}}{\tilde{g}}^3\cos {\tilde{g}}(L-l)\sin {\tilde{g}}\alpha \cos {\tilde{g}}(L-\alpha-l)+\frac{1}{2}{\tilde{f}}^3(L-l)\cos^2 {\tilde{g}}\alpha+\frac{1}{2}{\tilde{f}}{\tilde{g}}^2(L-l)\sin^2{\tilde{g}}\alpha\\
&&-\frac{1}{2}({\tilde{f}}^3-\tilde{f}{\tilde{g}}^2)\alpha \cos {\tilde{g}}(L-l) \cos {\tilde{g}}(L-2\alpha-l).
\end{array}
$$
 \end{enumerate}
 }
\end{Lemma}

\begin{Remark}\label{remarksec2}
{\rm
If Neumann boundary conditions are applied at $x=0$ and $x=L$, then by the fact of $\lambda_0(L)=0$ and Lemma \ref{lambdafg1} (or (\ref{lambdacomparison})), we have $\lambda_1(\alpha,l)+g'(0)<\lambda_0(L)=0$. Therefore, (H2) or (H3) cannot occur in this case.
}
\end{Remark}

\subsection{Population persistence under specific boundary
conditions}

In this subsection, we obtain precise conditions for population persistence for model \eqref{model1dge}-\eqref{udbcs} under different boundary conditions.

We first introduce the following relations that will be used in the analysis:
 \begin{equation}\label{parametersign1}
\begin{array}{l}
\alpha\geq0,\, l>0, \alpha+l\leq L, \, L-l\geq 0,\,
{\tilde{f}}>0,\, {\tilde{g}}>0,
\end{array}
\end{equation}
and \begin{equation}\label{parametersign2}
\begin{array}{l}
\mbox{ under (H1), } -1\leq R_1\leq 1,\,-1\leq R_2\leq 1,\,
A>0,\,
B>0,\, C>0, \, D<0,\,
 T_1>0, \,T_4>0,
\end{array}
\end{equation}
 where $\tilde{f}$, $\tilde{g}$, $A$, $B$, $C$, $D$, $T_1-T_4$, $R_1$, $R_2$,  $\hat{R_1}$, and $\hat{R_2}$ are defined in (\ref{titldefandg}), (\ref{R12def}), and  (\ref{ABCDT1T2def}).

{\bf Case 1. Neumann boundary conditions at $x=0, L$}.
By Remark \ref{remarksec2},  $\lambda_1(\alpha,l)+g'(0)<0$ and only Lemma \ref{lambdaalphataneqnsgl0} applies in this case.
Note that
\begin{equation}\label{NNABCDT1T2def}
\begin{array}{ll}
R_1=1,& R_2=e^{-2{\tilde{g}}L},\quad \hat{R_1}=0,\quad  \hat{R_2}=0,\\
A=e^{-\tilde{g}\alpha}+e^{\tilde{g}\alpha}, & B=e^{-\tilde{g}(\alpha+l)}+e^{-2{\tilde{g}}L}e^{\tilde{g}(\alpha+l)},\\
C=-e^{-\tilde{g}\alpha}+e^{\tilde{g}\alpha}, & D= -e^{-\tilde{g}(\alpha+l)}+
e^{-2{\tilde{g}}L}e^{\tilde{g}(\alpha+l)}.
\end{array}
\end{equation}
Moreover, (\ref{RRtanequ1}) becomes
\begin{equation}\label{NNtanequ}
\begin{array}{rl}
\tan\tilde{f}l
&=\dfrac{2\tilde{f}\tilde{g}\left(e^{-\tilde{g}l}-e^{-2\tilde{g}L}e^{\tilde{g}l}
\right)}{\tilde{f}^2(e^{-\tilde{g}\alpha}+e^{\tilde{g}\alpha})\left(e^{-\tilde{g}(\alpha+l)}+e^{-2\tilde{g}L}e^{\tilde{g}(\alpha+l)}\right)
+\tilde{g}^2(-e^{-\tilde{g}\alpha}+e^{\tilde{g}\alpha})\left(-e^{-\tilde{g}(\alpha+l)}+
e^{-2\tilde{g}L}e^{\tilde{g}(\alpha+l)}\right)},
\end{array}
\end{equation}
and (\ref{lambdaalphagl0}) becomes
\begin{equation}\label{NNlambdaalpha}
{\lambda_1}_\alpha=\frac{{\tilde{f}}{\tilde{g}}^2({\tilde{f}}^2
+{\tilde{g}}^2)T_1
T_3}{-(\frac{{\tilde{f}}}{2{\tilde{g}}}
+\frac{{\tilde{g}}}{2{\tilde{f}}})T_1T_4+ \frac{{\tilde{f}}({\tilde{f}}^2+{\tilde{g}}^2)}{2} T_1\alpha T_3
-2{\tilde{f}} (A^2{\tilde{f}}^2+C^2{\tilde{g}}^2)(L-l)e^{-2{\tilde{g}}L} -\frac{ l(T_2^2+{\tilde{f}}^2{\tilde{g}}^2 T_1^2)}{2{\tilde{f}}}}.
\end{equation}

We first determine the sign of $T_3$.
Note that $$T_3=BC+AD=2e^{-{\tilde{g}}(2L-2\alpha-l)}-2e^{-{\tilde{g}}(2\alpha+l)}.
$$ Then we have  $T_3<0$ if $\alpha\in [0,(L-l)/2)$, $T_3=0$ if $\alpha=(L-l)/2$, and $T_3>0$ if $\alpha\in((L-l)/2,L-l)$. By (\ref{parametersign1}) and (\ref{parametersign2}), this implies that
${\lambda_1}_\alpha>0$ when $\alpha\in [0,(L-l)/2)$, ${\lambda_1}_\alpha=0$ when $\alpha=(L-l)/2$.
To obtain the sign of ${\lambda_1}_\alpha$ when $\alpha>(L-l)/2$, we
need the following observation:
 \begin{equation}\label{lambda0lLlequal}
\lambda_1(\alpha,l)=\lambda_1(L-\alpha-l,l).
\end{equation}
This relation can be derived by using the variable change $y=L-x$ in (\ref{model1dgelineig}). Therefore, if ${\lambda_1}_\alpha>0$ when
$\alpha\in [0,(L-l)/2)$, then ${\lambda_1}_\alpha<0$ when
$\alpha\in ((L-l)/2,L-l]$.
We conclude the above results as follows.
\begin{Lemma}\label{monolambdaalphaNN}
{\rm
When Neumann boundary conditions are applied at $x=0,L$, $\lambda_1(\alpha,l)$
increases in $\alpha$ when $\alpha\in[0,(L-l)/2)$ and decreases in
$\alpha$ when $\alpha\in ((L-l)/2,L-l]$.
}
\end{Lemma}
It is obvious that $\lambda_1(\alpha,l)$ is continuous on $\alpha$.
For $\alpha=0$, \eqref{NNtanequ} becomes
\begin{equation}\label{NNequ-0}
\tan
\tilde{f}l|_{\alpha=0}=\tan\sqrt{f'(0)+\lambda_1(0,l)}l=\dfrac{\tilde{g}\left(e^{\tilde{g}(L-l)}-e^{\tilde{g}(l-L)}
\right)}{\tilde{f}\left(e^{\tilde{g}(L-l)}+e^{\tilde{g}(l-L)}
\right)}>0.
\end{equation}
Thus  $\lambda_1(0,l)$ satisfies
$0<\sqrt{f'(0)+\lambda_1(0,l)}l<\pi/2$, which implies $\lambda_1(0,l)<0$ if $f'(0)\geq \pi^2/(4l^2)$.
By (\ref{lambda0lLlequal}), we also have $\lambda_1(L-l,l)<0$ if $f'(0)\geq \pi^2/(4l^2)$.
For $\alpha=(L-l)/2$,
\begin{equation}\label{NNequ-Ll2n}
\tan \tilde{f}l|_{\alpha=\frac{L-l}{2}}=
\frac{2\tilde{f}\tilde{g}(-e^{-\tilde{g}(2L-l)}+e^{-\tilde{g}l})}
{(e^{-\tilde{g}(2L-l)}+e^{-\tilde{g}l})(\tilde{f}^2-\tilde{g}^2)+2e^{-\tilde{g}L}(\tilde{f}^2+\tilde{g}^2)}|_{\alpha=\frac{L-l}{2}}.
\end{equation}
Note that $\tilde{f}^2-\tilde{g}^2=f'(0)+g'(0)+2\lambda_1(\alpha,l)$. If $\lambda_1(\frac{L-l}{2},l)>0$, then $[\tilde{f}^2-\tilde{g}^2]|_{\alpha=\frac{L-l}{2}}>0$ as $f'(0)\geq -g'(0)$ and $(\tan \tilde{f}l)_{\alpha=\frac{L-l}{2}}>0$ since the right-hand side of (\ref{NNequ-Ll2n}) is positive, and hence $\tilde{f}l=\sqrt{f'(0)+\lambda_1(\frac{L-l}{2},l)}l\in (0,\pi/2)$. This cannot happen if $f'(0)\geq\pi^2/(4l^2)$.
Therefore, if $f'(0)\geq\pi^2/(4l^2)$, we always have $\lambda_1(\frac{L-l}{2},l)<0$.

By Lemma \ref{monolambdaalphaNN}, we have the following result.
\begin{Theorem}\label{persistenceNN}
{\rm
When Neumann boundary conditions are applied at $x=0,L$, the following statements hold.
\begin{enumerate}
  \item[(i)] If
   $f'(0)\geq\pi^2/(4l^2)$, then
      the population will persist no matter where the protection zone is;
      \item[(ii)]  If  $f'(0)<\pi^2/(4l^2)$ and $\lambda((L-l)/2,l)<0$,  then
                    the population will persist no matter where the protection zone is;

\item[(iii)]  If  $f'(0)<\pi^2/(4l^2)$ and
 $\lambda((L-l)/2,l)>0$,
  then
  \begin{enumerate}
    \item if $\lambda_1(0,l)> 0$, then
         the population will be extinct no matter where the protection zone is;
        \item if $\lambda_1(0,l)<0$, then the population will persist if  the protection zone is designed on $[0,l]$ or $[L-l,l]$.
  \end{enumerate}
  \end{enumerate}
  }
\end{Theorem}
\begin{proof}
Note that in (i) and (ii), $\lambda_1(\alpha,l)<0$ for all $\alpha\in[0,L-l]$ and that in (iii)(a), $\lambda_1(\alpha,l)>0$ for all $\alpha\in[0,L-l]$. The results follow from Theorem \ref{persistenceupt}.
\end{proof}

 Furthermore, notice the following facts.
\begin{itemize}
  \item If $f'(0)> \pi^2/(4L^2)$, then there exists $\bar{l}_1\in (0,L)$ such that $f'(0)\geq\pi^2/(4l^2)$ for $l\in (\bar{l}_1,L)$. If $\lambda_1(0,\bar{l}_1)<0$, then
                 there exists $\bar{l}_2\in (0,\bar{l}_1)$ such that $\lambda_1(0,l)
          <0$ for all $l\in (\bar{l}_2,\bar{l}_1)$ and
$\lambda_1(0,l)
>0$  for all $l\in (0,\bar{l}_2)$, since $\lambda_1(0,0)=-g'(0)>0$ and $\lambda_1(0,l)$ decreases in $l$.
  \item If $f'(0)\leq \pi^2/(4L^2)$, then $f'(0)<\pi^2/(4l^2)$ for all $l\in (0,L)$. Since $\lambda_1(0,L)=-f'(0)<0$ and $\lambda_1(0,l)$ decreases in $l$,  there exists $\tilde{l}\in (0,L)$ such that $\lambda_1(0,l)
<0$ for all $l\in (\tilde{l},L)$ and
$\lambda_1(0,l)
>0$  for all $l\in (0,\tilde{l})$.
 \end{itemize}
We then have the following observations for setting up a protection zone for population persistence.
\begin{Theorem}\label{finremarkNN}
{\rm
When Neumann boundary conditions are applied at $x=0,L$,  the following statements hold.
\begin{enumerate}
  \item[(i)] If $f'(0)> \pi^2/(4L^2)$, then
          \begin{enumerate}
         \item for $l\in (\bar{l}_1,L)$,
                   a protection zone
                   can ensure population persistence;

         \item if $\lambda_1(0,\bar{l}_1)<0$, then

\begin{enumerate}
  \item[(b-1)] for   $l\in (\bar{l}_2,\bar{l}_1)$,  a protection zone on $[0,l]$ or $[L-l,L]$ can ensure population persistence;

         \item[(b-2)]  for   $l\in (0,\bar{l}_2)$,
                  the population will be extinct;
\end{enumerate}
         \item if $\lambda_1(0,\bar{l}_1)\geq0$, then for  all $l\in (0,\bar{l}_1)$,
                           the population will be extinct.
       \end{enumerate}

  \item[(ii)]
If $f'(0)\leq \pi^2/(4L^2)$, then
\begin{enumerate}
  \item for  all $l\in (\tilde{l},L)$,  a protection zone on $[0,l]$ or $[L-l,L]$ can ensure population persistence;

         \item for  all $l\in (0,\tilde{l})$,
              the population will be extinct.
\end{enumerate}
 \end{enumerate}
 }
\end{Theorem}

 {\bf Case 2. Dirichlet boundary conditions at $x=0,\,L$}. In this case, $\lambda_0(L)=\pi^2/L^2$ and $\lambda_1(\alpha,l)+g'(0)<\lambda_0(L)=\pi^2/L^2$, so $\lambda_1(\alpha,l)+g'(0)$ could be $<0$, $=0$ or $>0$.
 In this case, (\ref{lambda0lLlequal}) remains true.
We also analyze $\lambda_1(\alpha,l)$ according to the sign of $\lambda_1(\alpha,l)+g'(0)$.

 If $\lambda_1(\alpha,l)+g'(0)<0$, then
\begin{equation}\label{DDABCDT1T2def}
\begin{array}{ll}
R_1=-1,& R_2=-e^{-2{\tilde{g}}L},\quad  \hat{R_1}=0,\quad   \hat{R_2}=0,\\
A=-e^{-\tilde{g}\alpha}+e^{\tilde{g}\alpha}, & B=e^{-\tilde{g}(\alpha+l)}-e^{-2{\tilde{g}}L}e^{\tilde{g}(\alpha+l)},\\
C=e^{-\tilde{g}\alpha}+e^{\tilde{g}\alpha}, & D= -e^{-\tilde{g}(\alpha+l)}-
e^{-2{\tilde{g}}L}e^{\tilde{g}(\alpha+l)}.
\end{array}
\end{equation}
Moreover, (\ref{RRtanequ1}) becomes
\begin{equation}\label{DDtanequ}
\begin{array}{l}
\tan
\tilde{f}l=\dfrac{2\tilde{f}\tilde{g}\left(e^{-\tilde{g}l}-e^{-2\tilde{g}L}e^{\tilde{g}l}
\right)}{-\tilde{g}^2(e^{-\tilde{g}\alpha}+e^{\tilde{g}\alpha})\left(e^{-\tilde{g}(\alpha+l)}+e^{-2\tilde{g}L}e^{\tilde{g}(\alpha+l)}\right)
-\tilde{f}^2(-e^{-\tilde{g}\alpha}+e^{\tilde{g}\alpha})\left(-e^{-\tilde{g}(\alpha+l)}+
e^{-2\tilde{g}L}e^{\tilde{g}(\alpha+l)}\right)},
\end{array}
\end{equation}
and (\ref{lambdaalphagl0}) becomes
\begin{equation}\label{DDlambdaalpha}
{\lambda_1}_\alpha=\frac{{\tilde{f}}{\tilde{g}}^2({\tilde{f}}^2
+{\tilde{g}}^2)T_1
T_3}{-(\frac{{\tilde{f}}}{2{\tilde{g}}}
+\frac{{\tilde{g}}}{2{\tilde{f}}})T_1T_4+ \frac{{\tilde{f}}({\tilde{f}}^2+{\tilde{g}}^2)}{2} T_1\alpha T_3
+2{\tilde{f}} (A^2{\tilde{f}}^2+C^2{\tilde{g}}^2)(L-l)e^{-2{\tilde{g}}L} -\frac{ l(T_2^2+{\tilde{f}}^2{\tilde{g}}^2 T_1^2)}{2{\tilde{f}}}}.
\end{equation}
Note that
$$
T_3=BC+AD=2e^{-g(2\alpha+l)}-2e^{-g(2L-2\alpha-l)}.
$$
Then we have  $T_3>0$ if $\alpha\in [0,(L-l)/2)$, $T_3=0$ if $\alpha=(L-l)/2$, and $T_3<0$ if $\alpha\in((L-l)/2,L-l)$.
By (\ref{parametersign1}) and (\ref{parametersign2}), this implies that ${\lambda_1}_\alpha=0$ when $\alpha=(L-l)/2$.
 By (\ref{DDlambdaalpha}) and the continuity of ${\lambda_1}_\alpha$ in $\alpha$ and $l$, we see that in the extreme case where $\alpha=0$ and
$l$ is sufficiently close to $L$, ${\lambda_1}_\alpha<0$. Since ${\lambda_1}_\alpha$ does not change
sign when $\alpha<(L-l)/2$, we obtain that ${\lambda_1}_\alpha<0$ for
$\alpha\in [0,(L-l)/2)$. By (\ref{lambda0lLlequal}), we then have ${\lambda_1}_\alpha>0$ for
$\alpha\in ((L-l)/2,L-l]$.

 If $\lambda_1(\alpha,l)+g'(0)=0$, then by (\ref{DDlambdaalphag0}), ${\lambda_1}_\alpha=0$ only when $\alpha=(L-l)/2$, ${\lambda_1}_\alpha<0$ when $\alpha=0$, and ${\lambda_1}_\alpha>0$ when $\alpha=L-l$. Since ${\lambda_1}_\alpha$ is continuous in $\alpha$, we know that ${\lambda_1}_\alpha<0$ for $\alpha\in (0,(L-l)/2)$ and ${\lambda_1}_\alpha>0$ for $\alpha\in ((L-l)/2,L-l)$.

If $\lambda_1(\alpha,l)+g'(0)>0$, then
   since $0<-g'(0)<\lambda_1(\alpha,l)<\lambda_0(L)-g'(0)=\pi^2/L^2-g'(0)$ and  $\tilde{g}^2=\lambda_1(\alpha,l)+g'(0)<\pi^2/L^2<\pi^2/(L-l)^2$, we have  $0<\tilde{g}(L-l)<\pi$. Since $0<\alpha<L-l$, we also have  $0<\tilde{g}(L-\alpha-l)<\pi$ and $\pi>\tilde{g}(L-\alpha-l)>\tilde{g}(L-2\alpha-l)>\tilde{g}(L-2(L-l)-l)=\tilde{g}(-L+l)>-\pi$. In particular, $\tilde{g}(L-2\alpha-l)\in (0,\pi)$ if $\alpha\in (0,(L-l)/2)$ and $\tilde{g}(L-2\alpha-l)\in (-\pi,0)$ if $\alpha\in ((L-l)/2,L-l)$. We also note that $ {\tilde{f}}{\tilde{g}}^4-{\tilde{f}}^3{\tilde{g}}^2=\tilde{f}\tilde{g}^2(\tilde{g}^2-\tilde{f}^2)
 =\tilde{f}\tilde{g}^2(\lambda_1(\alpha,l)+g'(0)-\lambda_1(\alpha,l)-f'(0))
=\tilde{f}\tilde{g}^2(g'(0)-f'(0))
 <0$. When $\alpha=0$ and $l$ is sufficiently close to $L$, it follows from (\ref{DDlambdaalphagg0}) that ${\lambda_1}_\alpha<0$, so ${\lambda_1}_\alpha<0$ for $\alpha\in[0,(L-l)/2)$. By (\ref{lambda0lLlequal}), we then have  ${\lambda_1}_\alpha>0$ for $\alpha\in((L-l)/2,L-l]$.

Based upon the above analysis, we have the following result.
\begin{Lemma}\label{monolambdaalphaDD}
{\rm
When Dirichlet boundary conditions are applied at $x=0,L$,  $\lambda_1(\alpha,l)$
decreases in $\alpha\in[0,(L-l)/2)$ and then increases in $\alpha\in((L-l)/2,L-l]$.
}
\end{Lemma}

Moreover,
if $ \lambda_1(\alpha,l)+g'(0)<0$,
for $\alpha=0$, \eqref{DDtanequ} becomes
\begin{equation}\label{DDequ-0}
\tan
\tilde{f}l|_{\alpha=0}=\tan\sqrt{f'(0)+\lambda_1(0,l)}l=-\dfrac{\tilde{f}\left(e^{\tilde{g}(L-l)}-e^{\tilde{g}(l-L)}
\right)}{\tilde{g}\left(e^{\tilde{g}(L-l)}+e^{\tilde{g}(l-L)}
\right)}<0.
\end{equation}
Thus  $\lambda_1(0,l)$ satisfies
$\pi/2<\sqrt{f'(0)+\lambda_1(0,l)}l<\pi$ or $\pi^2/(4l^2)-f'(0)<\lambda_1(0,l)<\pi^2/l^2-f'(0)$. This  implies $\lambda_1(0,l)<0$ if $f'(0)>\pi^2/l^2$ and $\lambda_1(0,l)>0$ if $f'(0)<\pi^2/(4l^2)$.
By (\ref{lambda0lLlequal}), $\lambda_1(L-l,l)<0$ if $f'(0)>\pi^2/l^2$ and $\lambda_1(L-l,l)>0$ if $f'(0)<\pi^2/(4l^2)$.
If $ \lambda_1(\alpha,l)+g'(0)\geq0$,
we have $\lambda_1(\alpha,l)\geq -g'(0)>0$. Note that $\lambda_1(\alpha,l)$ can be solved from (\ref{DDtanequg0}) or (\ref{DDtanequgg0}), which implies that $f'(0)<\pi^2/l^2$. We then obtain population persistence and extinction based on the information in the protection zone.

\begin{Theorem}\label{persistenceDD}
{\rm
When Dirichlet boundary conditions are applied at $x=0,L$, the following statements hold.
 \begin{enumerate}
   \item[(i)] If $f'(0)\geq \pi^2/l^2$, then
       the population will persist no matter where the protection zone is.
      \item[(ii)] If $f'(0)< \pi^2/l^2$ and $\lambda_1((L-l)/2,l)\geq 0$, then
                  the population will be extinct.
           \item[(iii)] If  $f'(0)< \pi^2/l^2$ and  $\lambda_1((L-l)/2,l)<0$, then the population will persist if the protection zone starts somewhere near $(L-l)/2$ and the optimal protection zone should be set up at  $[(L-l)/2,(L+l)/2]$.
 \end{enumerate}
 }
 \end{Theorem}

Furthermore, notice the following facts.
\begin{itemize}
  \item If $f'(0)>\pi^2/L^2$, then there exists $\bar{l}_1>0$ such that $f'(0)>\pi^2/l^2$ for all $l\in (\bar{l}_1,L)$. If $\lambda_1((L-\bar{l}_1)/2,\bar{l}_1)<0$, then there exists $\bar{l}_2\in (0,\bar{l}_1)$ such that $\lambda_1((L-l)/2,l)<0$ for all $l\in (\bar{l}_2,\bar{l}_1]$ and $\lambda_1((L-l)/2,l)>0$ for all $l\in (0,\bar{l}_2)$.
  \item If $ f'(0)\leq \pi^2/L^2$, then by (\ref{lambdacomparison}) and  Table \ref{lambda0Ltable}, we have
$$0\leq \frac{\pi^2}{L^2}-f'(0)=\lambda_0(L)-f'(0)=\lambda_{L,f}<\lambda_1(\alpha,l).
$$
That is, $\lambda_1(\alpha,l)>0$ is always true.
  \end{itemize}
We then have the following observations regarding the setup of the protection zone.
\begin{Theorem}\label{finremarkDD}
{\rm When Dirichlet boundary conditions are applied at $x=0,L$, the following statements hold.
\begin{enumerate}
  \item[(i)] If $f'(0)>\pi^2/L^2$, then
         \begin{enumerate}
         \item for $l\in (\bar{l}_1,L)$,
                     the population will persist no matter where the protection zone is;
         \item if $\lambda_1((L-\bar{l}_1)/2,\bar{l}_1)<0$, then
                 \begin{enumerate}
           \item[(b-1)]  for $l\in (\bar{l}_2,\bar{l}_1]$, a protection zone on $[(L-l)/2,(L+l)/2]$ can ensure population persistence;
           \item[(b-2)]  for $l\in (0,\bar{l}_2)$,
                        the population will be extinct;
         \end{enumerate}
          \item if $\lambda_1((L-\bar{l}_1)/2,\bar{l}_1)>0$, then
                         the population will be extinct.
                                      \end{enumerate}

  \item[(ii)] If $ f'(0)\leq \pi^2/L^2$, then
 the population will be extinct.
 \end{enumerate}
}
 \end{Theorem}

{\bf  Case 3. Neumann boundary condition at $x=0$ and Dirichlet boundary condition at $x=L$}. In this case, $\lambda_1(\alpha,l)+g'(0)<\lambda_0(L)=\pi^2/(4L^2)$, so $\lambda_1(\alpha,l)+g'(0)$ could be $<0$, $=0$ or $>0$.

If $ \lambda_1(\alpha,l)+g'(0)<0$,
then
\begin{equation}\label{NDABCDT1T2def}
\begin{array}{ll}
R_1=1,& R_2=-e^{-2{\tilde{g}}L},\quad \hat{R_1}=0,\quad  \hat{R_2}=0,\\
A=e^{-\tilde{g}\alpha}+e^{\tilde{g}\alpha}, & B=e^{-\tilde{g}(\alpha+l)}-e^{-2{\tilde{g}}L}e^{\tilde{g}(\alpha+l)},\\
C=-e^{-\tilde{g}\alpha}+e^{\tilde{g}\alpha}, & D= -e^{-\tilde{g}(\alpha+l)}-
e^{-2{\tilde{g}}L}e^{\tilde{g}(\alpha+l)}.
\end{array}
\end{equation}
Moreover, (\ref{RRtanequ1}) becomes
\begin{equation}\label{NDtanequ}
\begin{array}{rl}
\tan\tilde{f}l &=\dfrac{2\tilde{f}\tilde{g}\left(e^{-\tilde{g}l}+e^{-2\tilde{g}L}e^{\tilde{g}l}
\right)}{\tilde{f}^2(e^{-\tilde{g}\alpha}+e^{\tilde{g}\alpha})\left(e^{-\tilde{g}(\alpha+l)}-e^{-2\tilde{g}L}e^{\tilde{g}(\alpha+l)}\right)
+\tilde{g}^2(-e^{-\tilde{g}\alpha}+e^{\tilde{g}\alpha})\left(-e^{-\tilde{g}(\alpha+l)}-
e^{-2\tilde{g}L}e^{\tilde{g}(\alpha+l)}\right)},
\end{array}
\end{equation}
and (\ref{lambdaalphagl0}) becomes
\begin{equation}\label{NDlambdaalpha}
{\lambda_1}_\alpha=\frac{{\tilde{f}}{\tilde{g}}^2({\tilde{f}}^2
+{\tilde{g}}^2)T_1
T_3}{-(\frac{{\tilde{f}}}{2{\tilde{g}}}
+\frac{{\tilde{g}}}{2{\tilde{f}}})T_1 T_4+ \frac{{\tilde{f}}({\tilde{f}}^2+{\tilde{g}}^2)}{2} T_1\alpha T_3
+2{\tilde{f}} (A^2{\tilde{f}}^2+C^2{\tilde{g}}^2)(L-l)e^{-2{\tilde{g}}L} -\frac{ l(T_2^2+{\tilde{f}}^2{\tilde{g}}^2 T_1^2)}{2{\tilde{f}}}}.
\end{equation}
It follows from
 $$T_3=BC+AD=-2e^{-g(2\alpha+l)}-2e^{-g(2L-2\alpha-l)}<0$$
 that ${\lambda_1}_\alpha$ cannot be zero. Since
${\lambda_1}_\alpha$ is continuous, it does not change sign when
$\alpha\in[0,L-l]$. By (\ref{NDlambdaalpha}) and the continuity of ${\lambda_1}_\alpha$ on $\alpha$ and $l$, we see that when $\alpha=0$, in the extreme case where $l$ is sufficiently close to $L$,
${\lambda_1}_\alpha>0$.
Therefore,  ${\lambda_1}_\alpha$ is always greater than $0$.
That is, ${\lambda_1}$ increases in $\alpha\in[0,L-l]$.

If $ \lambda_1(\alpha,l)+g'(0)=0$, by (\ref{NDtanequg0}) and (\ref{NDlambdaalphag0}), $\tan\tilde{f}l>0$ and  ${\lambda_1}_\alpha>0$.

If $0<\lambda_1(\alpha,l)+g'(0)<\lambda_0(L)=\pi^2/(4L^2)$, then $\lambda_1(\alpha,l)<\pi^2/(4L^2)-g'(0)=\pi^2/(4L^2)-g'(0)$, which implies $\tilde{g}^2=\lambda_1(\alpha,l)+g'(0)<\pi^2/(4L^2)<\pi^2/(4(L-l)^2)$ and hence, $0<\tilde{g}(L-l)<\pi/2$. Since $0<\alpha<L-l$, we also have  $0<\tilde{g}(L-\alpha-l)<\pi/2$. Hence, by (\ref{NDtanequgg0}), we have  $\tan\tilde{f}l>0$. Furthermore, $\pi/2>\tilde{g}(L-\alpha-l)>\tilde{g}(L-2\alpha-l)>\tilde{g}(L-2(L-l)-l)=\tilde{g}(-L+l)>-\pi/2$, and ${\tilde{f}}^3{\tilde{g}}^2-{\tilde{f}}{\tilde{g}}^4
>0$. By (\ref{NDlambdaalphagg0}), we know that ${\lambda_1}_\alpha$ cannot be $0$. When $\alpha=0$ and $l=L$, (\ref{NDlambdaalphagg0}) implies ${\lambda_1}_\alpha>0$. Therefore, ${\lambda_1}_\alpha>0$ for all $\alpha\in [0,L-l]$.

Then we have the following result regarding the monotonicity of $\lambda_1(\alpha,l)$ in $\alpha$.
\begin{Lemma}\label{monolambdaalphaND}
{\rm
When Neumann boundary condition is applied at $x=0$ and Dirichlet boundary condition is applied at $x=L$, $\lambda_1(\alpha,l)$
increases in $\alpha\in[0,L-l]$.
}
\end{Lemma}

If $ \lambda_1(\alpha,l)+g'(0)<0$, when $\alpha=0$, \eqref{NDtanequ} becomes
\begin{equation}\label{NDequ-0}
\tan
\tilde{f}l|_{\alpha=0}=\tan\sqrt{f'(0)+\lambda_1(0,l)}l=-\dfrac{\tilde{g}\left(e^{\tilde{g}(L-l)}+e^{\tilde{g}(l-L)}
\right)}{\tilde{f}\left(-e^{\tilde{g}(L-l)}+e^{\tilde{g}(l-L)}
\right)}>0.
\end{equation}
Thus,
if $f'(0)\geq \pi^2/(4l^2)$, then $\lambda_1(0,l)<0$.
 When $\alpha=L-l$,
\begin{equation}\label{NDequ-Ll}
\tan
\tilde{f}l|_{\alpha=L-l}=\dfrac{\tilde{f}(e^{-\tilde{g}(L-l)}+e^{\tilde{g}(L-l)})e^{-\tilde{g}L}}{\tilde{g}(e^{-\tilde{g}(2L-l)}-e^{-\tilde{g}l})}<0. \end{equation}
Hence, $\lambda_1(L-l,l)$ satisfies $\pi/2<\sqrt{f'(0)+\lambda_1(L-l,l)}l<\pi$, or $\pi^2/(4l^2)-f'(0)<\lambda_1(L-l,l)<\pi^2/l^2-f'(0)$ . This  implies $\lambda_1(L-l,l)<0$ if $f'(0)\geq \pi^2/l^2$ and $\lambda_1(L-l,l)>0$ if $f'(0)<\pi^2/(4l^2)$.

 If $ \lambda_1(\alpha,l)+g'(0)\geq 0$, since $\tan\tilde{f}l >0$ and $\lambda_1(\alpha,l)>0$, we have $f'(0)<\pi^2/(4l^2)$.
By applying Lemma \ref{monolambdaalphaND}, we obtain the following results.
\begin{Theorem}\label{persistenceND}
{\rm When Neumann boundary condition is applied at $x=0$, and Dirichlet boundary condition is applied at $x=L$, the following statements hold.
 \begin{enumerate}
   \item[(i)] If $f'(0)\geq \pi^2/l^2$, then
         the population will persist no matter where the protection zone is.
   \item[(ii)]  If $\pi^2/(4l^2)<f'(0)<\pi^2/l^2$, then
       the population will persist if the protection zone starts from somewhere near $0$.
   \item[(iii)]  If $f'(0)\leq\pi^2/(4l^2)$, then
    the population will be extinct if the protection zone starts from somewhere near $L-l$.
 \end{enumerate}
 }
 \end{Theorem}
 \begin{proof}
Note that in (i)  $\lambda_1(\alpha,l)<0$ for all $\alpha\in[0,L-l]$, in (ii) $\lambda_1(0,l)<0$,
 and that in (iii), $\lambda_1(L-l,l)>0$.
   The results follow from Theorem \ref{persistenceupt}.
\end{proof}

Notice the following facts.
\begin{itemize}
  \item When $f'(0)>\pi^2/L^2$, there exists $\bar{l}_1\in (0,L)$ and $\bar{l}_2\in (0,\bar{l}_1)$ such that $f'(0)\geq\pi^2/l^2$ for all $l\in [\bar{l}_1,L)$, $\pi^2/(4l^2)< f'(0)<\pi^2/l^2$ for all $l\in (\bar{l}_2,\bar{l}_1)$, and $ f'(0)<\pi^2/(4l^2)$ for all $l\in (0,\bar{l}_2)$. If $ \lambda_1(0,\bar{l}_2)<0$, then there exists $\bar{l}_3\in (0,\bar{l}_2)$ such that $ \lambda_1(0,l)<0$ for all $l\in(\bar{l}_3,\bar{l}_2]$ and $ \lambda_1(0,l)>0$ for all $l\in (0,\bar{l}_3)$.
  \item When $\pi^2/(4L^2)<f'(0)\leq\pi^2/L^2$, there exists $\bar{l}_2\in (0,L)$ such that $\pi^2/(4l^2)<f'(0)\leq\pi^2/l^2$ for all $l\in (\bar{l}_2,L)$ and $ f'(0)<\pi^2/(4l^2)$ for all $l\in (0,\bar{l}_2)$.
  \item When  $f'(0)\leq\pi^2/(4L^2)$,
          then by (\ref{lambdacomparison}) and  Table \ref{lambda0Ltable}, we have
$$0\leq \frac{\pi^2}{4L^2}-f'(0)=\lambda_0(L)-f'(0)=\lambda_{L,f}<\lambda_1(\alpha,l).
$$
\end{itemize}
 We then have the following observations.
 \begin{Theorem}\label{finremarkND}
 {\rm When Neumann boundary condition is applied at $x=0$, and Dirichlet boundary condition is applied at $x=L$, the following results hold.
 \begin{enumerate}
 \item[(i)] If $f'(0)>\pi^2/L^2$, then
       \begin{enumerate}
         \item for  $l\in [\bar{l}_1,L)$,
                   the population will persist no matter where the protection zone is;
             \item for   $l\in (\bar{l}_2,\bar{l}_1)$,
                      the population will persist if the protection zone starts near $x=0$ and
                     the optimal protection zone should be set up at $[0,l]$;
         \item if $ \lambda_1(0,\bar{l}_2)<0$, then
                \begin{enumerate}
         \item[(c-1)]  for $l\in(\bar{l}_3,\bar{l}_2]$, the population will persist if the protection zone is set up at $[0,l]$;
             \item[(c-2)]  for  $l\in(0,\bar{l}_3]$, the population will be extinct;
          \end{enumerate}
        \item if $ \lambda_1(0,\bar{l}_2)\geq 0$, then for   $l\in (0,\bar{l}_2)$,
         the population will  be extinct.
       \end{enumerate}
   \item[(ii)]  If $\pi^2/(4L^2)<f'(0)\leq\pi^2/L^2$, then
       the results in (i)(b)-(i)(d) are valid by replacing $\bar{l}_1$ with $L$ in (i)(b).

   \item[(iii)]   If  $f'(0)\leq\pi^2/(4L^2)$, then
        the population will be extinct no matter where the protection zone is.
  \end{enumerate}
 }
\end{Theorem}

 {\bf Case 4. Dirichlet boundary condition at $x=0$ and Neumann boundary condition at $x=L$}. In this case, $\lambda_0(L)=\pi^2/(4L^2)$ and $\lambda_1(\alpha,l)+g'(0)<\lambda_0(L)=\pi^2/(4L^2)$.

Note that by letting $y=L-x$, one can rewrite the eigenvalue problem (\ref{model1dgelineig}) with Dirichlet Boundary condition at $x=0$ and Neumann Boundary condition at $x=L$ into an eigenvalue problem in the same form of (\ref{model1dgelineig}) with Neumann Boundary condtion at $x=0$ and Dirichlet Boundary condition at $x=L$. Thus, by adapting the results in Case 3, we can obtain
${\lambda_1}_\alpha<0$ in this case and in turn the following results hold.

\begin{Lemma}\label{monolambdaalphaDN}
{\rm When Dirichlet boundary condition is applied at $x=0$, and Neumann boundary condition is applied at $x=L$, $\lambda_1(\alpha,l)$
decreases in $\alpha\in[0,L-l]$.
}
\end{Lemma}

\begin{Theorem}\label{persistenceDN}
{\rm When Dirichlet boundary condition is applied at $x=0$, and Neumann boundary condition is applied at $x=L$, the following statements hold.
\begin{enumerate}
  \item[(i)] If $f'(0)\geq\pi^2/l^2$, then
     the population will persist no matter where the protection zone is.
  \item[(ii)] If $\pi^2/(4l^2)<f'(0)<\pi^2/l^2$, then
      the population will persist if the protection zone starts somewhere near $L-l$.
  \item[(iii)] If $f'(0)\leq \pi^2/(4l^2)$, then
    the population will be extinct if the protection zone starts near $0$.
\end{enumerate}
}
 \end{Theorem}
 Furthermore, the following results are valid.
\begin{itemize}
  \item When $f'(0)>\pi^2/L^2$, there exists $\bar{l}_1\in (0,L)$ and $\bar{l}_2\in (0,\bar{l}_1)$ such that $f'(0)\geq\pi^2/l^2$ for all $l\in [\bar{l}_1,L)$, $\pi^2/(4l^2)< f'(0)<\pi^2/l^2$ for all $l\in (\bar{l}_2,\bar{l}_1)$,  and $ f'(0)<\pi^2/(4l^2)$ for all $l\in (0,\bar{l}_2)$. If $ \lambda_1(L-\bar{l}_2,\bar{l}_2)<0$, then there exists $\bar{l}_3\in (0,\bar{l}_2)$ such that $ \lambda_1(L-l,l)<0$ for all $l\in(\bar{l}_3,\bar{l}_2]$ and $ \lambda_1(L-l,l)>0$ for all $l\in (0,\bar{l}_3)$.
  \item When $\pi^2/(4L^2)<f'(0)\leq\pi^2/L^2$, there exists $\bar{l}_2\in (0,L)$ such that $\pi^2/(4l^2)<f'(0)\leq\pi^2/l^2$ for all $l\in (\bar{l}_2,L)$ and $ f'(0)<\pi^2/(4l^2)$ for all $l\in (0,\bar{l}_2)$.
  \item When  $f'(0)\leq\pi^2/(4L^2)$,
          then by (\ref{lambdacomparison}) and  Table \ref{lambda0Ltable}, we have
 $\lambda_1(\alpha,l)>0$ for all $ \alpha\in [0,L-l]$.
\end{itemize}
\begin{Theorem}\label{finremarkDN}
{\rm When Dirichlet boundary condition is applied at $x=0$, and Neumann boundary condition is applied at $x=L$, the following statements hold.
\begin{enumerate}
 \item[(i)] If $f'(0)>\pi^2/L^2$, then
       \begin{enumerate}
       \item for $l\in [\bar{l}_1,L)$,
            the population will persist no matter where the protection zone is;
             \item for  $l\in (\bar{l}_2,\bar{l}_1)$,
                       the population will persist if the protection zone starts near $x=L-l$ and
                    the optimal protection zone should be set up at $[L-l,L]$;

         \item if $ \lambda_1(L-\bar{l}_2,\bar{l}_2)<0$, then
                  \begin{enumerate}
         \item[(c-1)]  for  $l\in(\bar{l}_3,\bar{l}_2]$, the population will persist if the protection zone is set up at $[L-l,L]$;
             \item[(c-2)]  for  $l\in(0,\bar{l}_3]$, the population will  be extinct;
          \end{enumerate}
        \item if $ \lambda_1(L-\bar{l}_2,\bar{l}_2)\geq 0$, then for   $l\in (0,\bar{l}_2)$,
         the population will be extinct.
       \end{enumerate}
   \item[(ii)]  If $\pi^2/(4L^2)<f'(0)\leq\pi^2/L^2$, then
    the results in (i)(b)-(i)(d) are valid by replacing $\bar{l}_1$ with $L$ in (i)(b).

   \item[(iii)]   If $f'(0)\leq\pi^2/(4L^2)$, then
      the population will be extinct.
     \end{enumerate}
}
 \end{Theorem}

From Cases 1-4, we see that if the boundary conditions only involve Neumann type and Dirichlet type, then to help population persist in the habitat, the protection zone should be close to the end with Neumann boundary condition but away from the end with Dirichlet condition;
 the longer the protection zone is the easier it is to help population persist;  if the growth rate in the protection zone is too low ($f'(0)<\pi^2/L^2$ in Case 2 or $f'(0)<\pi^2/(4L^2)$ in Cases 3 and 4) then a protection zone may not help population persist in the whole habitat.

\vskip4pt
{\bf Case 5. Robin boundary conditions ($\alpha_1\varphi'(0)-\alpha_2\varphi(0)=0$) at $x=0$ or ($\beta_1\varphi'(L)+\beta_2\varphi(L)=0$) at $x=L$}.

When the Robin boundary condition is applied at $x=0$ or at $x=L$, by analyzing the sign of $\tan
\tilde{f}l$ in (\ref{RRtanequ1}) and the sign of ${\lambda_1}_\alpha$ in (\ref{lambdaalphagl0}), we can obtain the following results.

\begin{figure}[t!]
\centering
\includegraphics[height=2in,
width=2in]{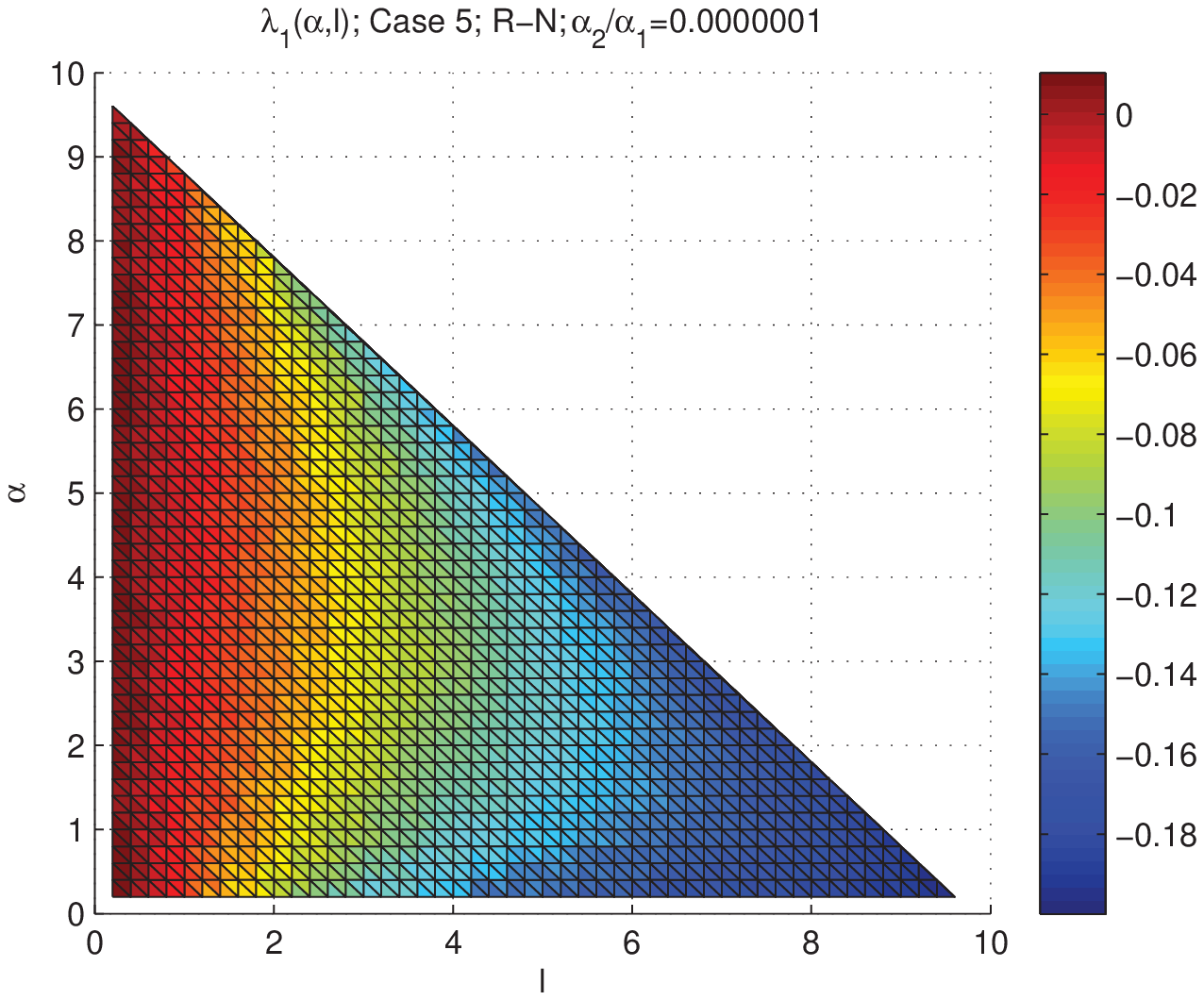}
\includegraphics[height=2in,
width=2in]{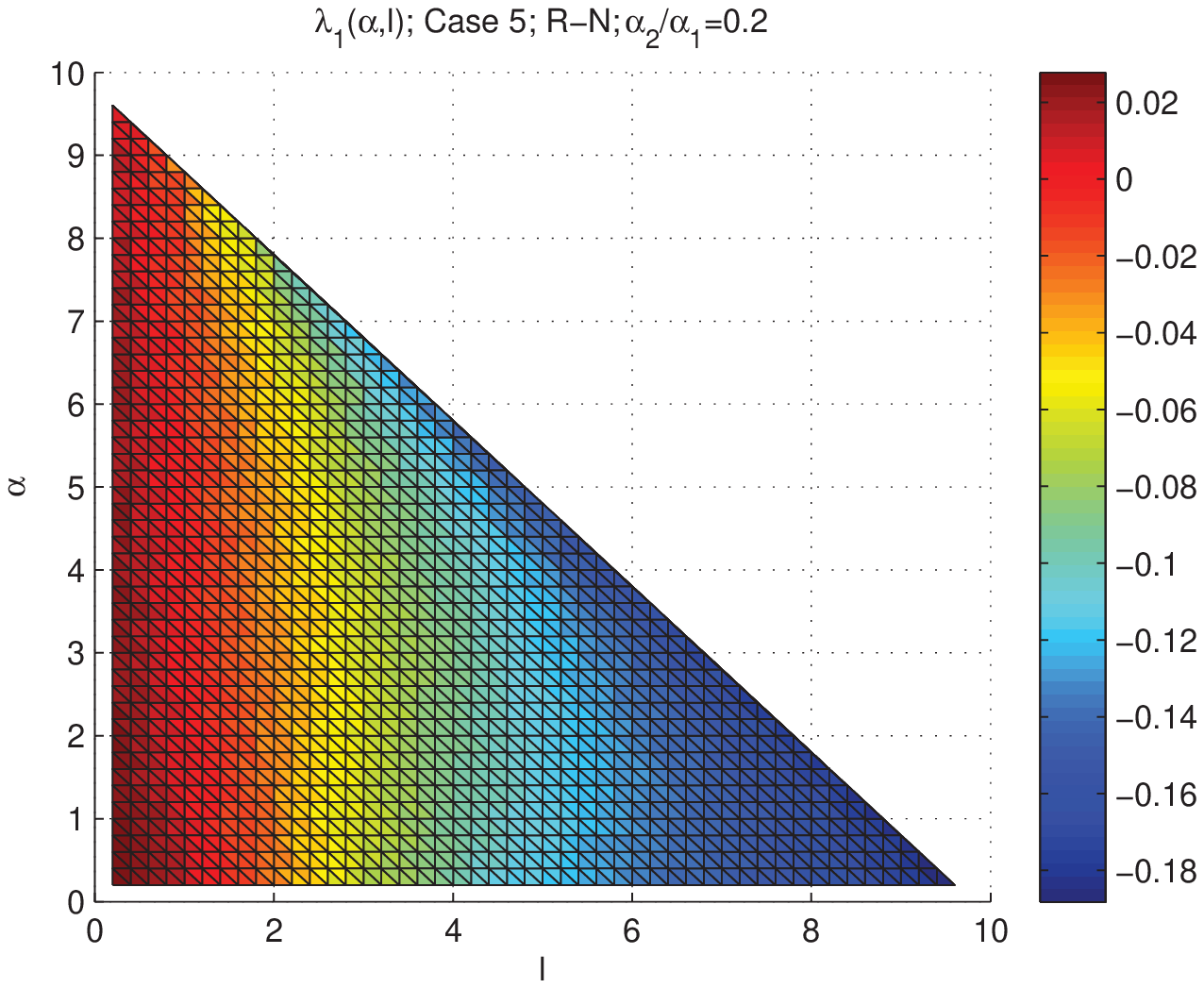}
\includegraphics[height=2in,
width=2in]{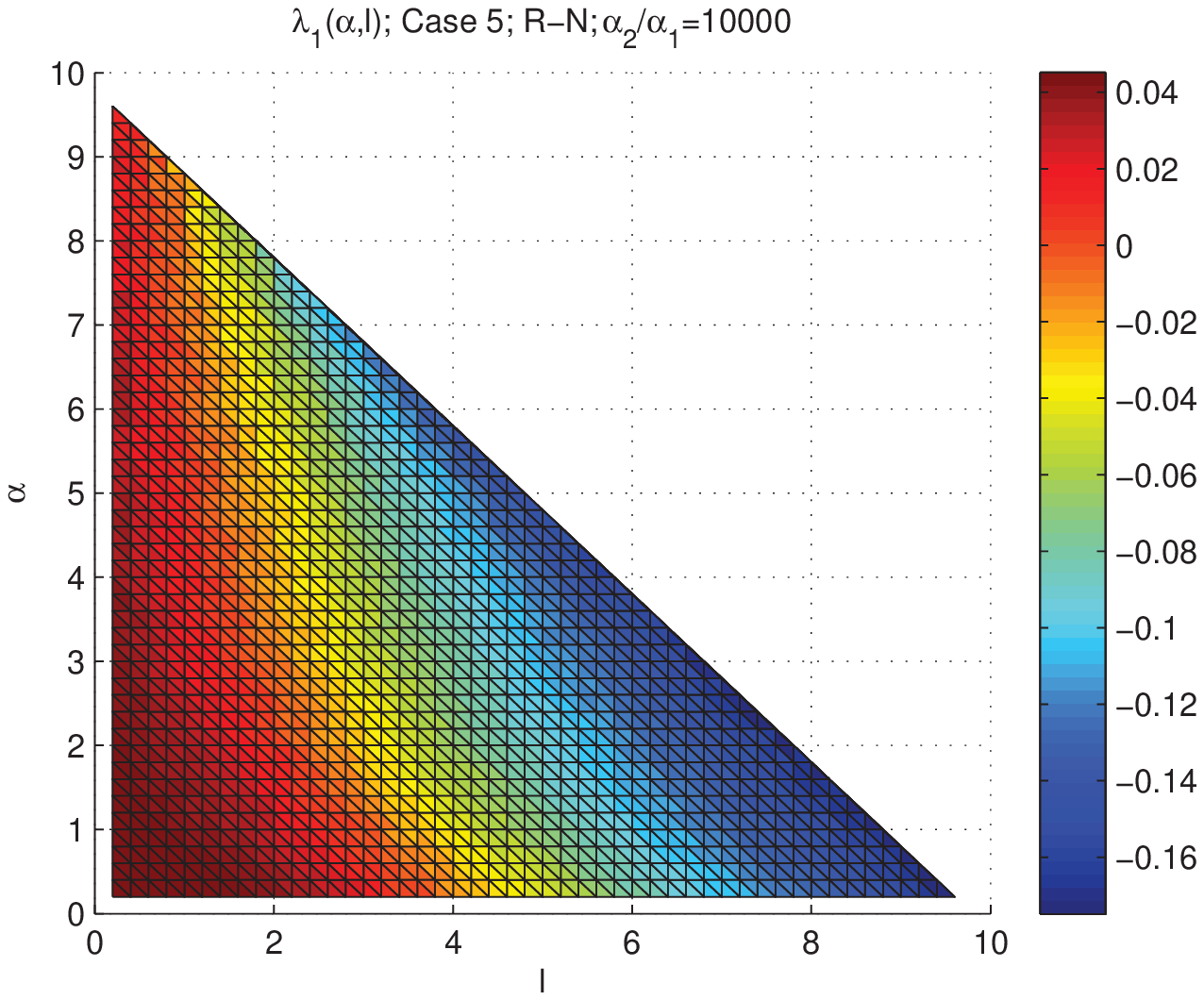}
\caption{The relation among $\lambda_1(\alpha,l)$ and $\alpha$ and $l$ in Case 5. In the figures, ``R-N" represents that the Robin boundary condition ($\alpha_1\varphi'(0)-\alpha_2\varphi(0)=0$) is applied at $x=0$ and the Neumann boundary condition is applied at $x=L$. Here $f(u)=ru(1-u)$ and $g(u)=ru(1-u)(u-a)$. Parameters: $r=0.2$, $a=0.1$, $L=10$.} \label{lambda1alphalRN}
\end{figure}

\begin{Theorem}\label{Thcases6to8}
{\rm The following statements hold.
\begin{enumerate}
\item[(i)] When Robin boundary condition ($\alpha_1\varphi'(0)-\alpha_2\varphi(0)=0$) is applied at $x=0$ and Neumann boundary condition is applied at $x=L$,
if  $\alpha_1>0$, $\alpha_2>0$, and $\alpha_2/\alpha_1$ is sufficiently small, then
 the results in Theorem \ref{persistenceNN} are true;
if  $\alpha_1>0$, $\alpha_2>0$, and $\alpha_2/\alpha_1$ is sufficiently large, then
 the results in Theorem \ref{persistenceDN} are true.
\item[(ii)]  When Robin boundary condition ($\alpha_1\varphi'(0)-\alpha_2\varphi(0)=0$) is applied at $x=0$ and Dirichlet boundary condition is applied at $x=L$,
if  $\alpha_1>0$, $\alpha_2>0$, and $\alpha_2/\alpha_1$ is sufficiently small, then
 the results in Theorem \ref{persistenceND} are true;
if  $\alpha_1>0$, $\alpha_2>0$, and $\alpha_2/\alpha_1$ is sufficiently large, then
 the results in Theorem \ref{persistenceDD} are true.
\item[(iii)]  When Neumann boundary condition is applied at $x=0$ and Robin boundary condition ($\beta_1\varphi'(L)+\beta_2\varphi(L)=0$) is applied at $x=L$,
if  $\beta_1>0$, $\beta_2>0$, and $\beta_2/\beta_1$ is sufficiently small, then
 the results in Theorem \ref{persistenceNN} are true;
if  $\beta_1>0$, $\beta_2>0$, and $\beta_2/\beta_1$ is sufficiently large, then
 the results in Theorem \ref{persistenceND} are true.
  \item[(iv)]  When Dirichlet boundary condition is applied at $x=0$ and Robin boundary condition  ($\beta_1\varphi'(L)+\beta_2\varphi(L)=0$) is applied at $x=L$,
if  $\beta_1>0$, $\beta_2>0$, and $\beta_2/\beta_1$ is sufficiently small, then
 the results in Theorem \ref{persistenceDN} are true;
if  $\beta_1>0$, $\beta_2>0$, and $\beta_2/\beta_1$ is sufficiently large, then
 the results in Theorem \ref{persistenceDD} are true.
\end{enumerate}
}
\end{Theorem}

 Figures \ref{lambda1alphalRN} and \ref{lambda1alphalRD} show the dependence of $ \lambda_1(\alpha,l)$ on $\alpha$ and $l$ for different ratios of $\alpha_2/\alpha_1$ when Robin boundary condition ($\alpha_1\varphi'(0)-\alpha_2\varphi(0)=0$) is applied at $x=0$ and Neumann or Dirichlet boundary condition is applied at $x=L$. The phenomena in the figures confirm the results in
 Theorem \ref{Thcases6to8}.
\begin{figure}[t!]
\centering
\includegraphics[height=2in,
width=2in]{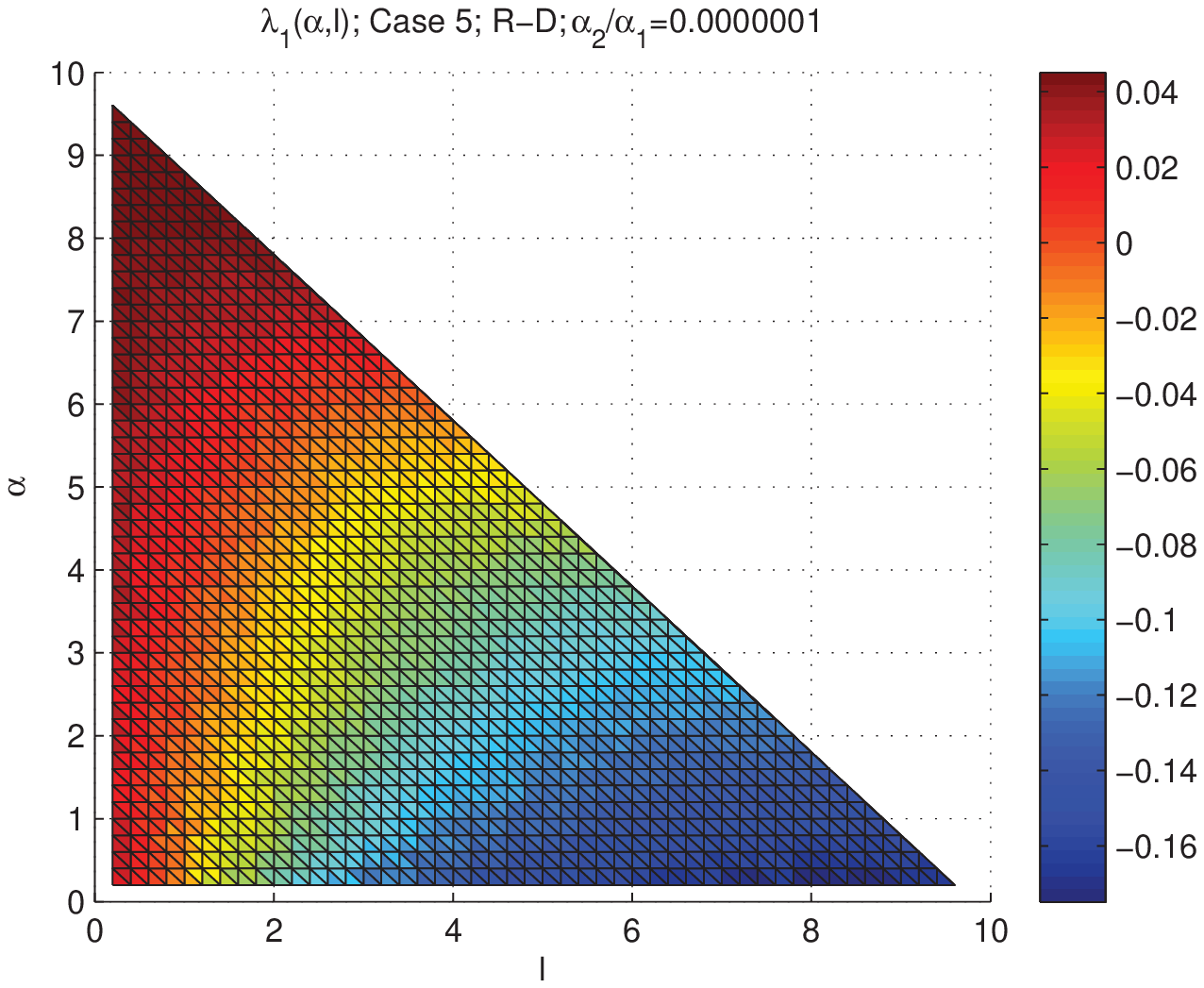}
\includegraphics[height=2in,
width=2in]{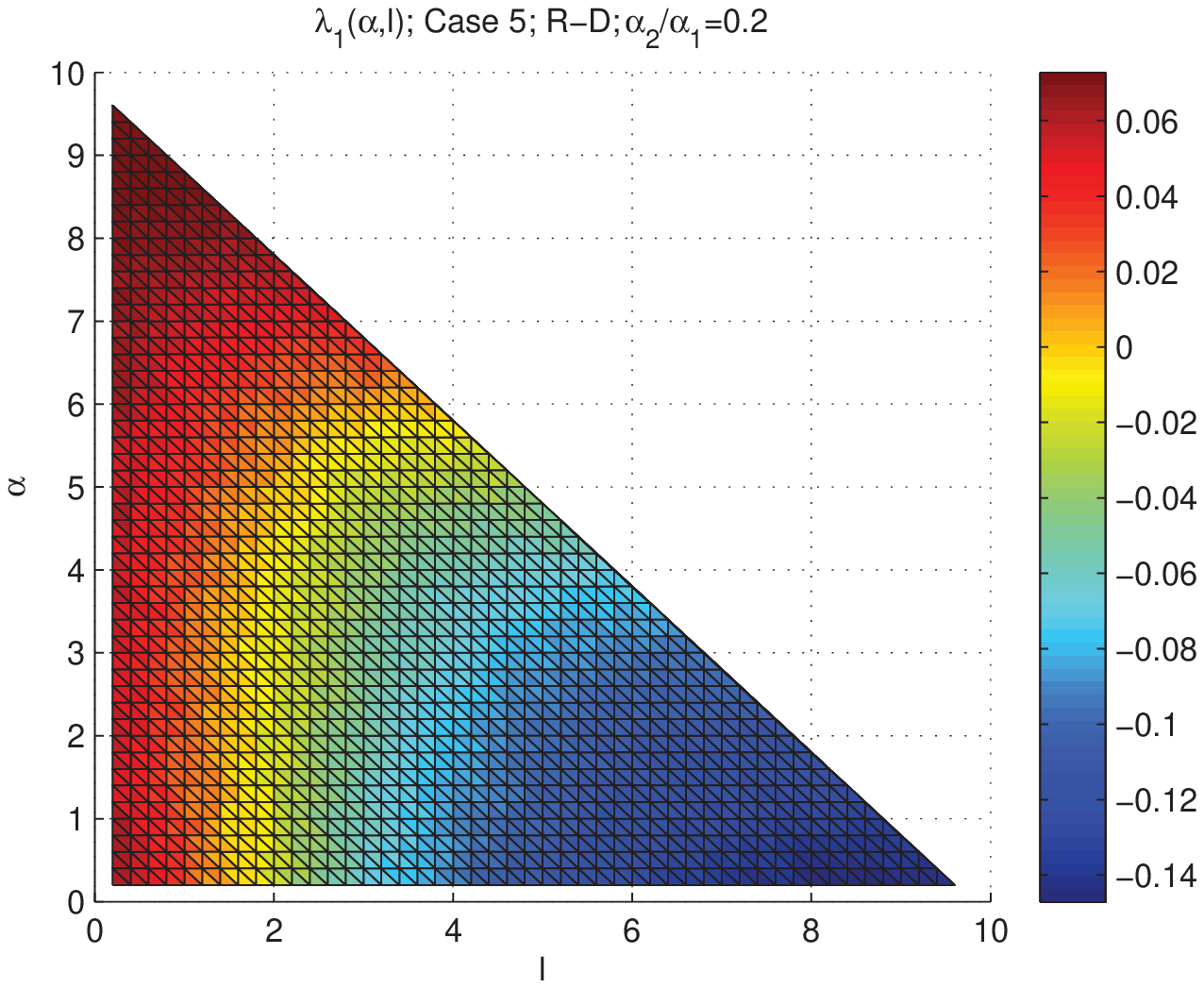}
\includegraphics[height=2in,
width=2in]{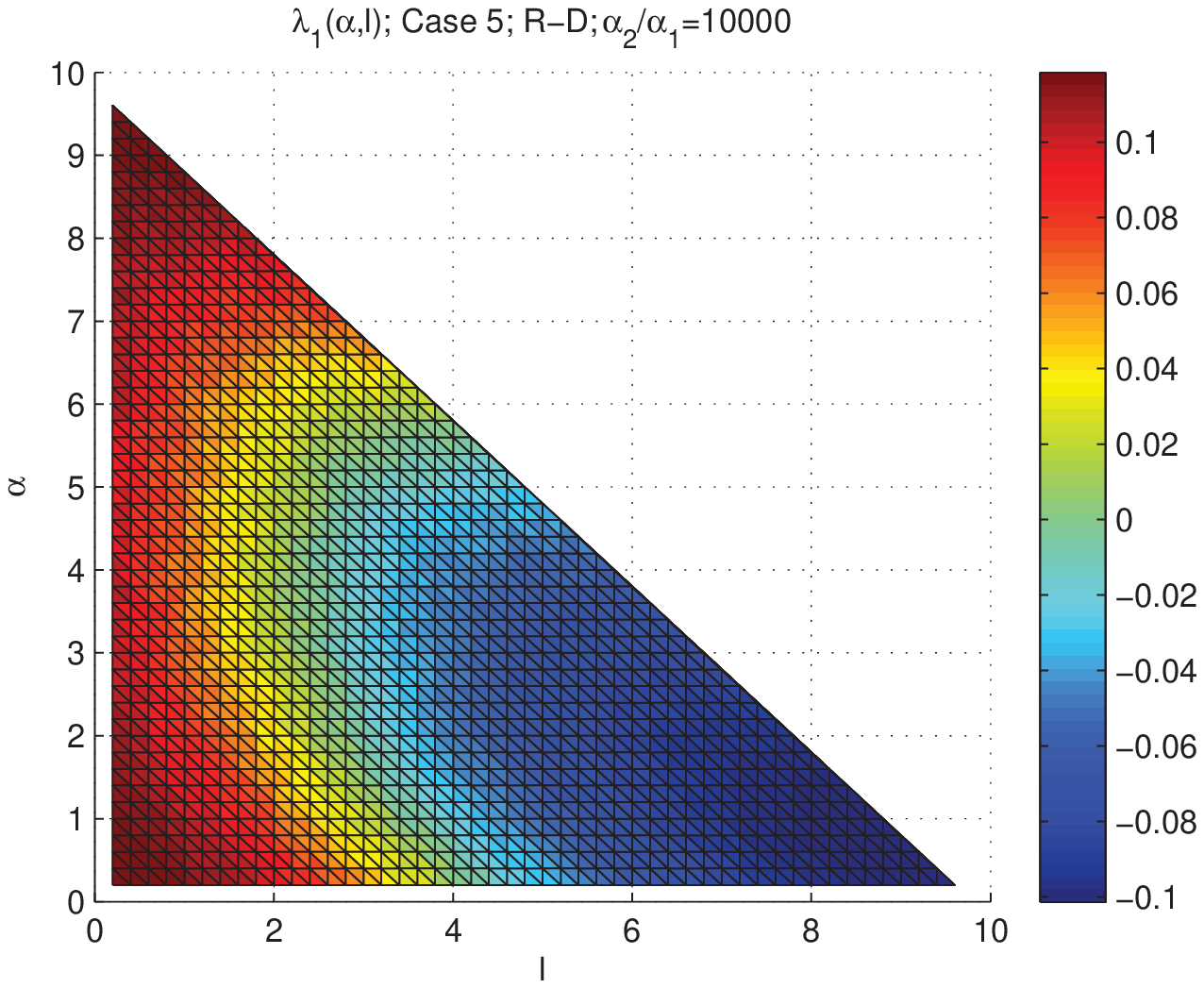}
\caption{The relation among $\lambda_1(\alpha,l)$ and $\alpha$ and $l$ in Case 5. ``R-D" represents that the Robin boundary condition ($\alpha_1\varphi'(0)-\alpha_2\varphi(0)=0$) is applied at $x=0$ and the Dirichlet boundary condition is applied at $x=L$.
Other parameters are the same as in Figure \ref{lambda1alphalRN}.
} \label{lambda1alphalRD}
\end{figure}

Figure \ref{solutiondiffbcs} shows the time evolution of a solution of \eqref{model1dge}-\eqref{udbcs} under different boundary conditions and with different protection zone locations. It indicates that when the initial population density is low, the same setup of the protection zone may lead to population persistence or extinction under different boundary conditions and  a different setup of the protection zone may help population persist in the whole domain even though the boundary conditions don't change. In particular, with the parameters in Figure \ref{solutiondiffbcs}, we see that
with the protection zone on the interval $[1,4]$ in the total habitat $[0,10]$, the population can persist if the boundary condition is Neumann at $x=0$ but it will be extinct if the boundary condition is Dirichlet at $x=0$; see Figure \ref{solutiondiffbcs} (a-b) and (d-e). If the boundary condition is Dirichlet at $x=0$, then shifting the protection zone from the interval $[1,4]$ to the interval $[3,6]$ helps the population persist; see Figure \ref{solutiondiffbcs} (b-c) and (e-f).

\begin{figure}[t!]
\centering
\subfigure[]{ \label{left}\hspace{-0.2cm}\includegraphics[height=2in,
width=2in]{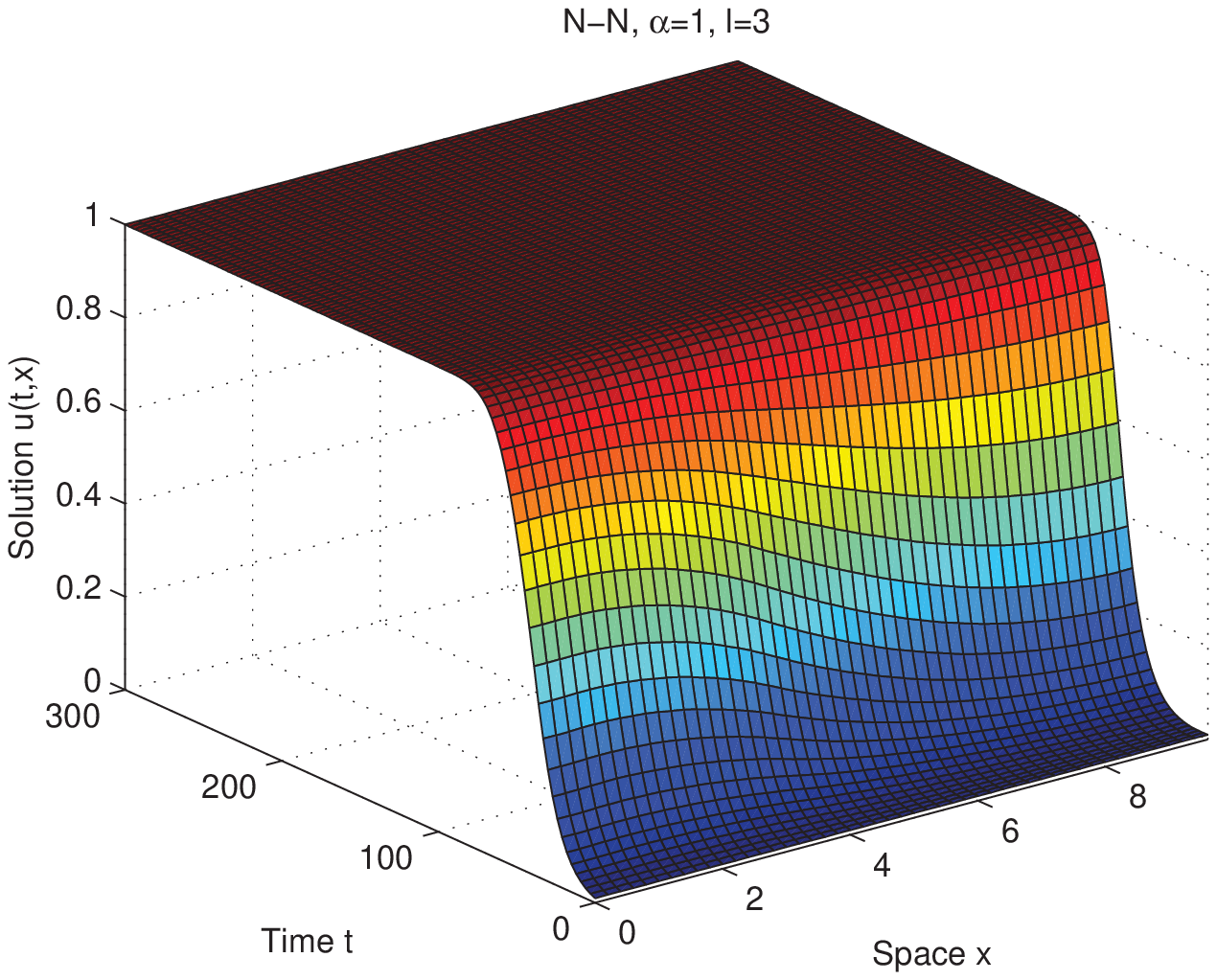}}
\hspace{-0.2cm}
  \subfigure[]{\label{right}\includegraphics[height=2in,
width=2in]{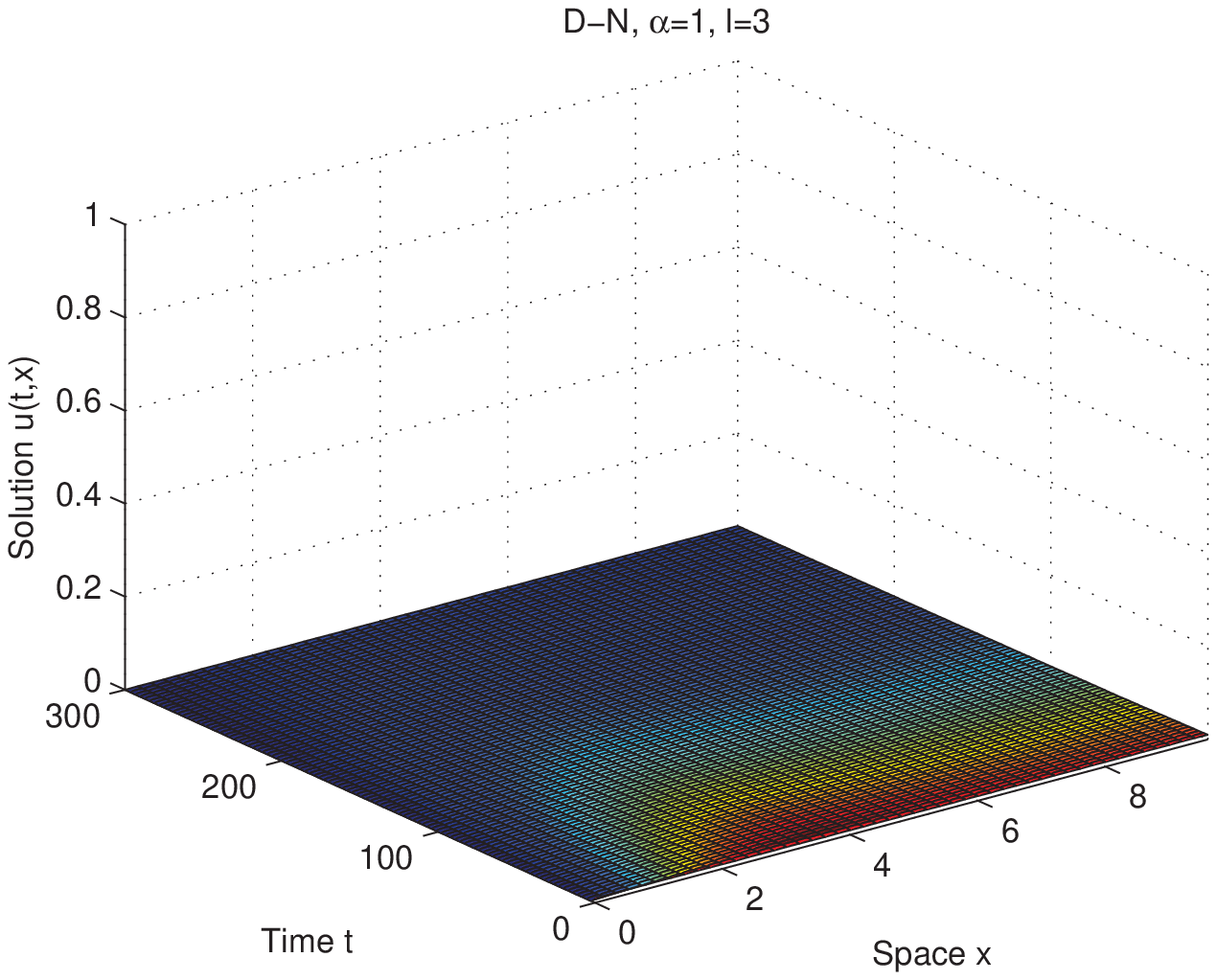}}
   \subfigure[]{\label{left}\hspace{-0.2cm} \includegraphics[height=2in,
width=2in]{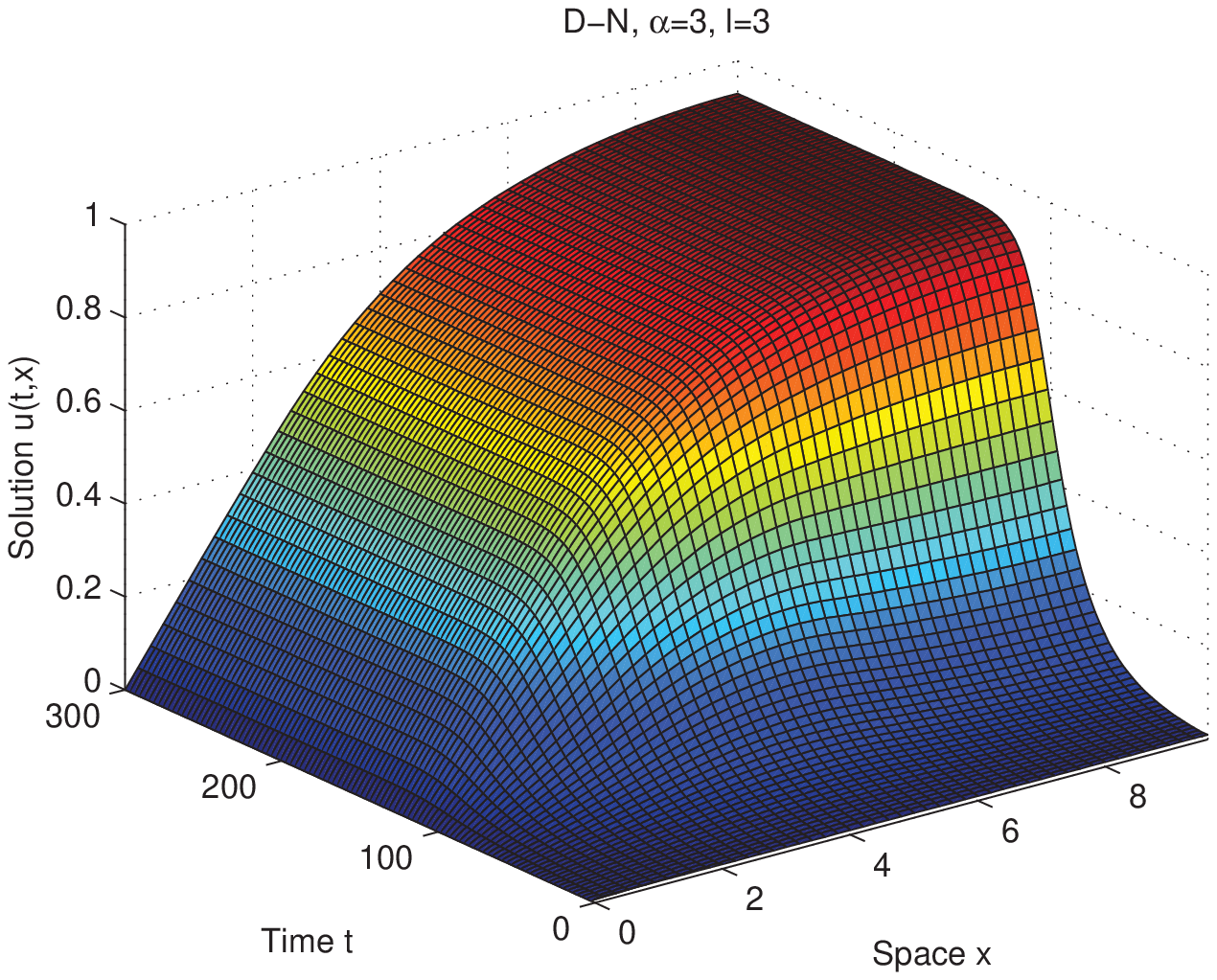}}\\
   \subfigure[]{\label{right} \hspace{-0.2cm}
   \includegraphics[height=2in,
width=2in]{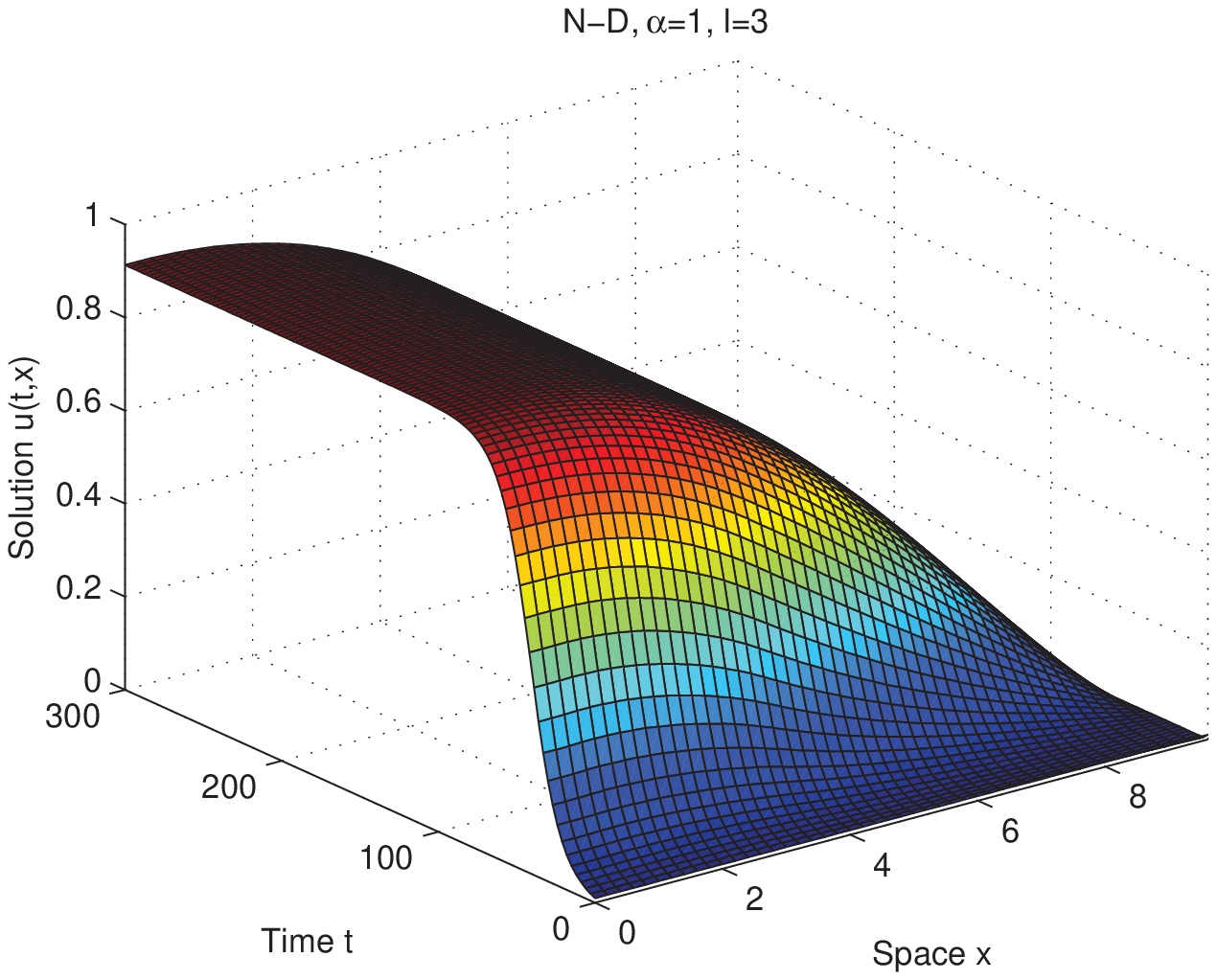}}
 \subfigure[]{\label{right} \hspace{-0.2cm}\includegraphics[height=2in,
width=2in]{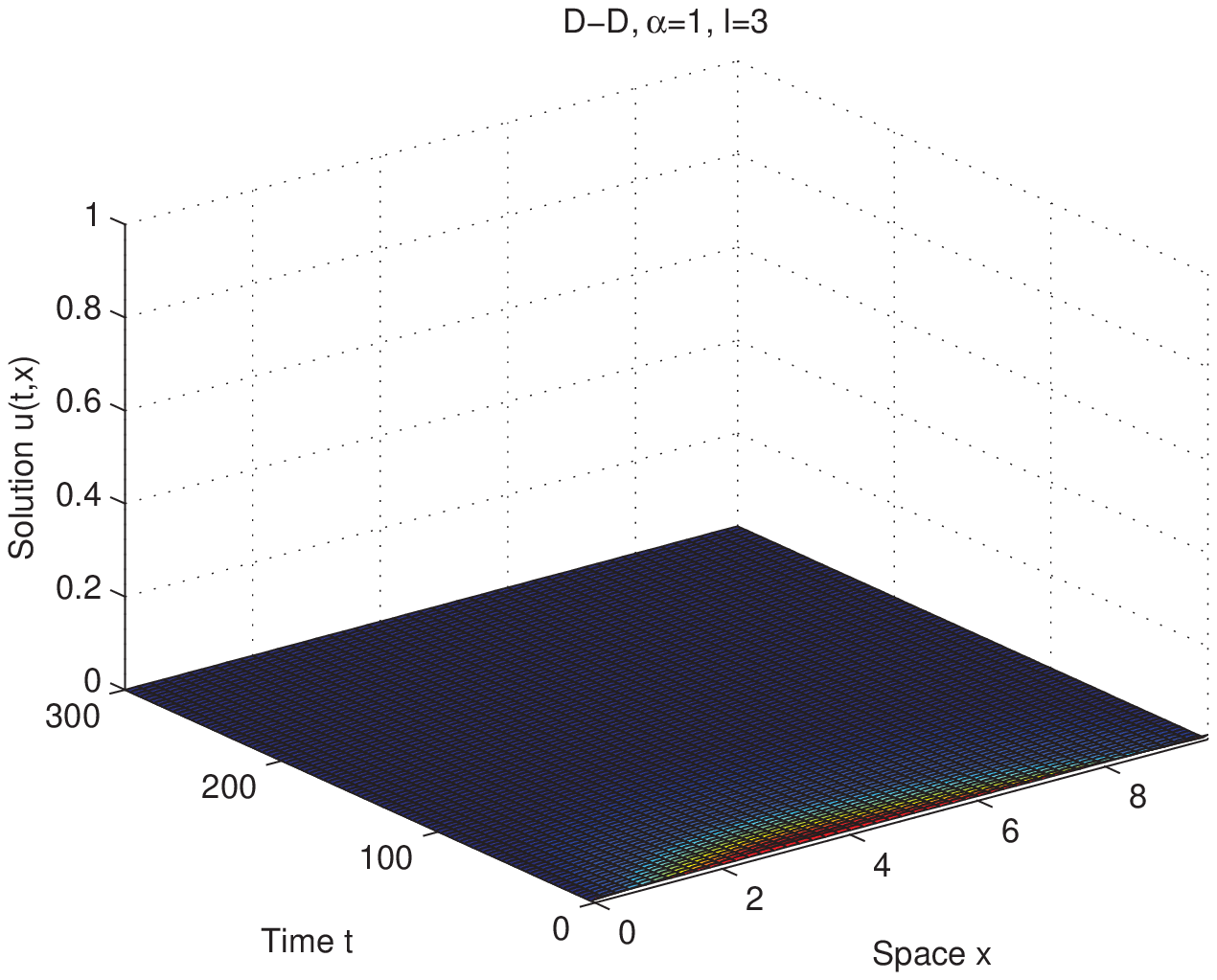}}
 \subfigure[]{\label{right} \hspace{-0.2cm}\includegraphics[height=2in,
width=2in]{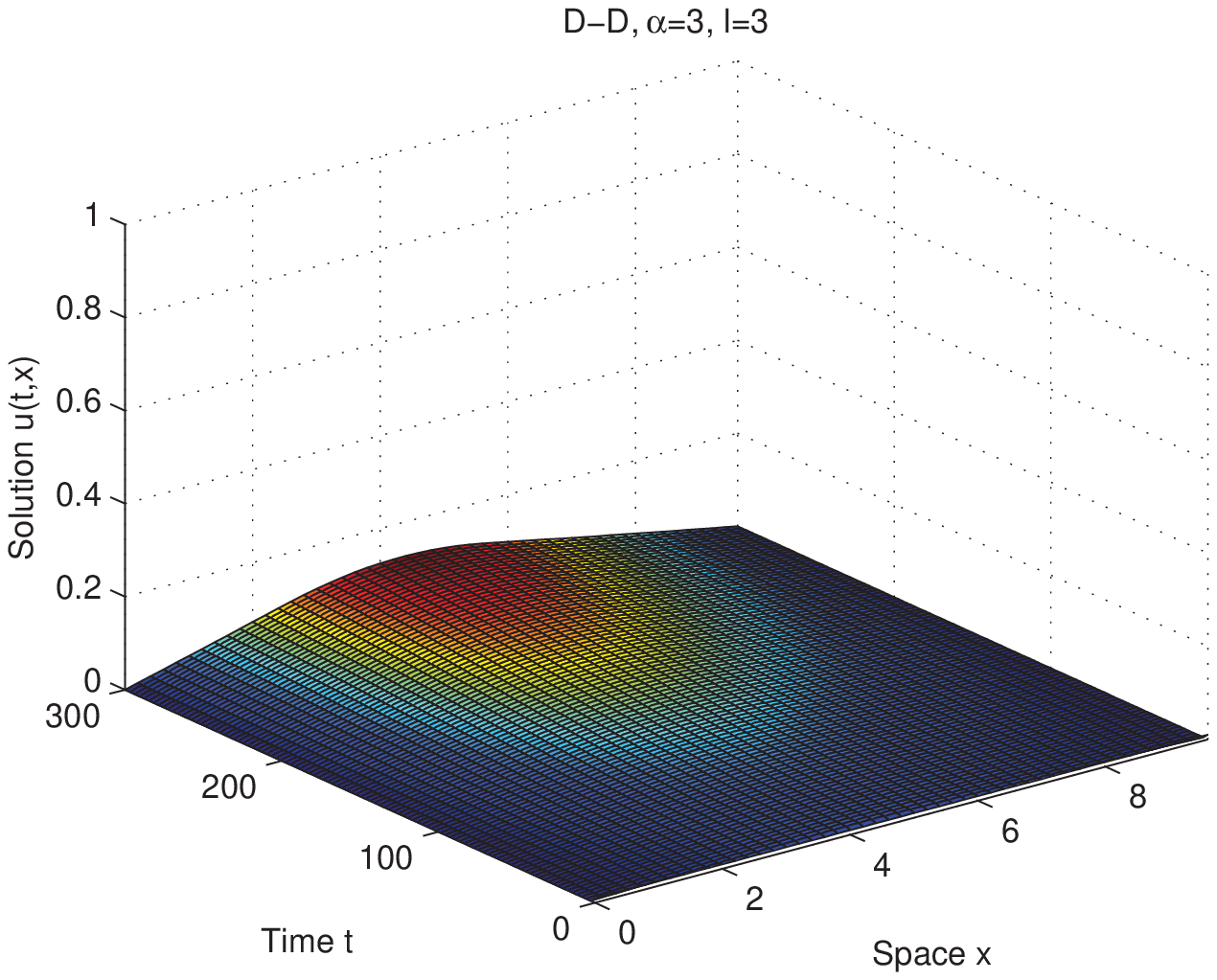}}
 \caption{The solutions of \eqref{model1dge}-\eqref{udbcs} under different boundary conditions and with different protection zone locations.
  Other parameters are the same as in Figure \ref{lambda1alphalRN}. The initial conditions are $u_0(x)\equiv 0.01$.}
 \label{solutiondiffbcs}
\end{figure}

\section{Discussion}

Designing effective strategies to protect an endangered species is
important for the conservation of the species itself
 as well as the conservation of the habitat and the health of
  the involved ecosystems. Protection zones for an endangered or native species have been observed to provide many economic,
social, environmental, and cultural values \cite{Holt1977,Holt1984,Ami2005,Jerry2010,Loisel,Halpern}. Mathematical models have been developed
  to investigate the feasibility or efficiency of protection zones.
  Previous studies have been mainly focused on qualitative analysis and  revealed that the size of a protection zone is always positively related to the persistence status of the population but when the size of the protection zone is small other factors such as the demographic or dispersal parameters also play important roles on the persistence of the population (see e.g., \cite{Dushi2006,Duliang2008,Dupengwang2009}).

   In this work,
  we consider a reaction-diffusion model for a species in a one-dimensional bounded
   domain $[0,L]$. We assume that the growth of the species in
 its natural domain is subjected to a strong Allee effect, that is, its growth rate is negative when the population density is low, and hence, the population will not be able to recover once its density is below a certain level and the species will eventually be extinct. To protect the species, we assume that there is a protection zone $[\alpha,\alpha+l]$ in the habitat where the population's growth satisfies a logistic type  function and the population can grow to its carrying capacity there.
 We have proved that the principal eigenvalue $\lambda_1(\alpha,l)$ of the  eigenvalue problem associated with the linearization of the system at the trivial solution can be used to determine population persistence ($\lambda_1(\alpha,l)<0$) and extinction ($\lambda_1(\alpha,l)>0$) (see Theorem \ref{persistenceupt}). Our goal is to specify
the relations between $\lambda_1(\alpha,l)$ and the parameters related to the protection zone and then provide precise strategies for an optimal protection zone in order for the population to persist.

  It is not surprising to obtain that $\lambda_1(\alpha,l)$ is always a decreasing function of the length of the protection zone $l$ (see Lemma \ref{lambdadeplth}), which implies that the longer the protection zone is, the easier it is for the
population to persist in the whole habitat. This is intuitively understandable and coincides with previous findings (see e.g., \cite{Dushi2006,Duliang2008,Dupengwang2009,Dupeng2019}).
The dependence of $\lambda_1(\alpha,l)$ on the starting point $\alpha$ of the protection zone is described in terms of the derivative of  $\lambda_1(\alpha,l)$  with respect to $\alpha$, denoted by ${\lambda_1}_\alpha$, but the sign of ${\lambda_1}_\alpha$ is complicated
and dependent of the boundary conditions.
We study it in the cases of different combinations of Neumann and Dirichlet conditions. It turns out that ${\lambda_1}_\alpha$ is positive (resp. negative) for all $\alpha\in [0,L-l)$ if the boundary conditions are Neumann (at $x=0$)-Dirichlet (at $x=L$) (resp. Dirichlet-Neumann) and that  ${\lambda_1}_\alpha$ is positive (resp. negative) first and then negative (resp. positive) when $\alpha$ increases from $0$ to $L-l$ if the boundary conditions are Neumann-Neumann (resp. Dirichlet-Dirichlet); see Lemmas \ref{monolambdaalphaNN}, \ref{monolambdaalphaDD},
 \ref{monolambdaalphaND}, and \ref{monolambdaalphaDN}.

By using the dependence of $\lambda_1(\alpha,l)$ on $\alpha$ and $l$, we then can investigate the effect of the protection zone on population persistence and extinction. In fact, if the growth rate in the protection zone is sufficiently large, then the population will persist in the whole domain in a long run regardless of the initial distribution and the location of the protection zone provided that the length of the protection zone is sufficiently large (see Theorems \ref{finremarkNN}(i)(a), \ref{finremarkDD}(i)(a), \ref{finremarkND}(i)(a), \ref{finremarkDN}(i)(a)); if the growth rate in the protection zone is sufficiently small, then there are always initial data (typically all sufficiently small initial data) such that the population will be extinct no matter where
the protection zone is or how long the protection zone is  (see Theorems  \ref{finremarkDD}(ii), \ref{finremarkND}(iii), \ref{finremarkDN}(iii));
otherwise, the population may persist when the protection zone is sufficiently long and is set up at some optimal locations  (see Theorems \ref{finremarkNN}((i)(b-1), (ii)(a)), \ref{finremarkDD}(i)(b-1), \ref{finremarkND}((i)(b), (i)(c-1), (ii)), \ref{finremarkDN}((i)(b), (i)(c-1), (ii))),
while the population will be extinct if the protection zone is not long enough or  is set up at a bad location  (see Theorems \ref{finremarkNN}((i)(b-2), (i)(c), (ii)(b)), \ref{finremarkDD}((i)(b-2), (i)(c), (ii)), \ref{finremarkND}((i)(c-2), (i)(d), (ii)), \ref{finremarkDN}((i)(c-2), (i)(d), (ii))). Note that in this paper, by ``extinction", we mean that the population will be extinct if its initial distribution density is sufficiently low.
 By comparing these results, we see that, in general, if there is a Neumann boundary end, then the optimal location for the protection zone, if needed, is the part of the domain connecting the Neumann boundary end; if both ends are Dirichlet, then the optimal location for the protection zone, if needed, should be in the middle of the domain.
 Due to the mathematical complexity rising in the cases of Robin boundary conditions, we have not obtained all detailed results about the optimal locations for the protection zone in these cases, but we have seen that if the Robin boundary condition is more like a Neumann (or Dirichlet) condition, then the results related to the system is more like those related to the system with a Neumann (or Dirichlet) condition.

We complete this paper by proposing several interesting problems that deserve further consideration.
The first one is to extend the study of the effects of protection zones on spatial population dynamics to the time-periodic setting.
Many plant and animal species have demonstrated seasonal population dynamics in response
to seasonal fluctuations and time-periodic varying environments, in particular, the weather conditions \cite{WangOgdenZhu}.
The spatiotemporal environments play an important role in characterizing the persistence and extinction of some species, such as diseases transmission \cite{Allen,Pengzhao2012,WangZhao}, single population growth \cite{CC1989}. Some plants may become extinct in arid seasons, while vegetate densely in moist climate, so it is interesting to investigate how protection zones can prevent population from extinction though the species lives in some arid seasons.
Another issue is to concern the protection mechanism in a network.
A real river system usually has rich topological structures, which greatly influence the population growth
and spread of organisms living in it \cite{Grant}. It is spontaneous and crucial to prevent endangered species in real river systems from extinction. From a mathematical point of view, the addition of advection to reaction-diffusion models may change the long
term outcome (persistence/extinction) of the population \cite{HsuLou,JinLewis} in river systems. Based on a local river reaction-diffusion equation, \cite{JinPengShi,DuLouPengZhou,Rami} studied the dynamical behavior of species spreading from a location in a river network where two or three branches meet. It has been found that both the water flow speeds in the river branches and the cross section areas can affect the extinction/persistence of species. Nevertheless, to our knowledge, there hasn't been any understanding about the long-term  population dynamics when a protection measure is considered to ensure the perpetual persistence in a river network. This should also be an interesting problem for future work.

\section{Appendix}

\subsection{Proof of Lemma \ref{lambdaalphataneqnsgeq0}}\label{applemmalambdaalphataneqnsgeq0}

Assume that  {\bf (H2)} is true, i.e., $g'(0)+\lambda_1(\alpha,l)=0.
$
We can write the general form of $\varphi_1$ in (\ref{model1dgelineig}) on $[0,L]$ as
\begin{equation}\label{varphi1solutiong0ap}
\varphi_1(x)=\left\{
\begin{array}{ll}
C_3+C_4x, &
\text{for}\;\;x\in[0,\alpha],\\
C_1\cos\tilde{f}x+C_2\sin\tilde{f}x, &
\text{for}\;\;x\in[\alpha,\alpha+l],\\
C_5+C_6x, &
\text{for}\;\;x\in[\alpha+l,L],
\end{array}
\right.
\end{equation}
where each $C_i$ ($i=1,\cdots,6$) is a constant. Then
\begin{equation*}
\varphi_1'(x)=\left\{
\begin{array}{ll}
C_4, &
\text{for}\;\;x\in[0,\alpha],\\
-C_1 {\tilde{f}} \sin\tilde{f}x+C_2 {\tilde{f}} \cos\tilde{f}x, &
\text{for}\;\;x\in[\alpha,\alpha+l],\\
 C_6, &
\text{for}\;\;x\in[\alpha+l,L].
\end{array}
\right.
\end{equation*}

We first consider the case where the Dirichlet boundary condition is applied at $x=0$ and $x=L$.
The boundary conditions of $\varphi_1$ at $0$ and $L$ lead to
$$
\begin{array}{l}
-\alpha_2 C_3=0,\quad
\beta_2 (C_5+C_6L)=0,
\end{array}
$$
which imply $C_3=0$ and $C_6=-C_5/L$.
Due to the continuous differentiability of $\varphi_1$ at $\alpha$ and $\alpha+l$, we infer that
\begin{equation}\label{RR1ap}
C_4 \alpha=C_1\cos\tilde{f}\alpha+C_2\sin\tilde{f}\alpha,
\end{equation}
\begin{equation}\label{RR3ap}
C_4=-C_1 {\tilde{f}}\sin\tilde{f}\alpha+C_2 {\tilde{f}} \cos\tilde{f}\alpha,
\end{equation}
\begin{equation}\label{RR2ap}
C_5-C_5(\alpha+l)/L=C_1\cos\tilde{f}(\alpha+l)+C_2\sin\tilde{f}(\alpha+l),
\end{equation}
\begin{equation}\label{RR4ap}
-C_5/L
=-C_1 {\tilde{f}}\sin\tilde{f}(\alpha+l)+C_2 {\tilde{f}}\cos\tilde{f}(\alpha+l).
\end{equation}
Then
$\eqref{RR1ap}\times\eqref{RR2ap}+\dfrac{\eqref{RR3ap}\times\eqref{RR4ap}}{\tilde{f}^2}$
leads to
\begin{equation}\label{RR01ap}
\begin{array}{l}
\D C_4 C_5(\alpha(1-\frac{\alpha+l}{L})-\frac{1}{L\tilde{f}^2})
= (C_1^2+C_2^2)\cos {\tilde{f}}l
\end{array}
\end{equation}
and $\eqref{RR1ap}\times\eqref{RR4ap}-\eqref{RR3ap}\times\eqref{RR2ap}$ leads to
\begin{equation}\label{RR02ap}
\begin{array}{l}
C_4C_5(-\frac{\alpha}{L}-(1-\frac{\alpha+l}{L}))
= -(C_1^2+C_2^2){\tilde{f}} \sin {\tilde{f}}l.
\end{array}
\end{equation}
By simplifying (\ref{RR01ap}) and (\ref{RR02ap}), we obtain
that $\lambda_1(\alpha,l)$ satisfies (\ref{DDtanequg0}).
Differentiating
 (\ref{DDtanequg0}) with respect $\alpha$ on two sides yields
$$
\begin{array}{rcl}
&&\sec^2 {\tilde{f}}l \cdot l \cdot
\D \frac{\lambda_{1\alpha}}{2\tilde{f}}
\\ &=&\D \frac{(L-l)\frac{\lambda_{1\alpha}}{2\tilde{f}}({\tilde{f}}^2\alpha(L-\alpha-l)-1)-(L-l){\tilde{f}}(2{\tilde{f}}\frac{\lambda_{1\alpha}}{2\tilde{f}}\alpha(L-\alpha-l)+{\tilde{f}}^2(L-\alpha-l)-{\tilde{f}}^2\alpha)}{T_2^2},
\end{array}
$$
which gives (\ref{DDlambdaalphag0}).

We next consider the case where the Neumann boundary condition is applied at $x=0$ and the Dirichlet boundary condition is applied at $x=L$.
The boundary conditions of $\varphi_1$ at $0$ and $L$ infer
$$
\begin{array}{l}
\alpha_1 C_4=0,\quad
\beta_2 (C_5+C_6L)=0,
\end{array}
$$
which imply $C_4=0$ and $C_6=-C_5/L$.
As $\varphi_1$ is continuously differentiable at $\alpha$ and $\alpha+l$, we have
\begin{equation}\label{RR1apND}
C_3=C_1\cos\tilde{f}\alpha+C_2\sin\tilde{f}\alpha,
\end{equation}
\begin{equation}\label{RR3apND}
0=-C_1 {\tilde{f}}\sin\tilde{f}\alpha+C_2 {\tilde{f}} \cos\tilde{f}\alpha,
\end{equation}
\begin{equation}\label{RR2apND}
C_5-C_5(\alpha+l)/L=C_1\cos\tilde{f}(\alpha+l)+C_2\sin\tilde{f}(\alpha+l),
\end{equation}
\begin{equation}\label{RR4apND}
-C_5/L
=-C_1 {\tilde{f}}\sin\tilde{f}(\alpha+l)+C_2 {\tilde{f}}\cos\tilde{f}(\alpha+l).
\end{equation}
Then
$\eqref{RR1apND}\times\eqref{RR2apND}+\dfrac{\eqref{RR3apND}\times\eqref{RR4apND}}{\tilde{f}^2}$
leads to
\begin{equation}\label{RR01apND}
\begin{array}{l}
\D C_3 C_5(1-\frac{\alpha+l}{L})
= (C_1^2+C_2^2)\cos {\tilde{f}}l
\end{array}
\end{equation}
and $\eqref{RR1apND}\times\eqref{RR4apND}-\eqref{RR3apND}\times\eqref{RR2apND}$ leads to
\begin{equation}\label{RR02apND}
\begin{array}{l}
-C_3C_5\frac{1}{L}
= -(C_1^2+C_2^2){\tilde{f}} \sin {\tilde{f}}l.
\end{array}
\end{equation}
By simplifying (\ref{RR01apND}) and (\ref{RR02apND}), we obtain
that $\lambda_1(\alpha,l)$ satisfies (\ref{NDtanequg0}).
Differentiating  (\ref{NDtanequg0})
 with respect $\alpha$ on two sides yields
$$
\begin{array}{rcl}
\sec^2 {\tilde{f}}l \cdot l \cdot
\frac{\lambda_{1\alpha}}{2\tilde{f}}=
\frac{1}{(L-\alpha-l)^2}\frac{1}{\tilde{f}}-\dfrac{1}{L-\alpha-l}\frac{\lambda_{1\alpha}}{2\tilde{f}^3},
\end{array}
$$
which implies (\ref{NDlambdaalphag0}) immediately.

\subsection{Proof of Lemma \ref{lambdaalphataneqnsgg0}}\label{applemmalambdaalphataneqnsgg0}

Assume that {\bf (H3)} is true, i.e.,  $g'(0)+\lambda_1(\alpha,l)>0.
$
We can write the general form of $\varphi_1$ on $[0,L]$ as
\begin{equation}\label{varphi1solutionapH3}
\varphi_1(x)=\left\{
\begin{array}{ll}
C_3\cos\tilde{g}x+C_4\sin\tilde{g}x, &
\text{for}\;\;x\in[0,\alpha],\\
C_1\cos\tilde{f}x+C_2\sin\tilde{f}x, &
\text{for}\;\;x\in[\alpha,\alpha+l],\\
C_5\cos\tilde{g}x+C_6\sin\tilde{g}x, &
\text{for}\;\;x\in[\alpha+l,L],
\end{array}
\right.
\end{equation}
where each $C_i$ ($i=1,\cdots,6$) is a constant.
Then it holds that
\begin{equation*}
\varphi_1'(x)=\left\{
\begin{array}{ll}
-C_3 {\tilde{f}} \sin\tilde{f}x+C_4 {\tilde{f}} \cos\tilde{f}x, &
\text{for}\;\;x\in[0,\alpha],\\
-C_1 {\tilde{f}} \sin\tilde{f}x+C_2 {\tilde{f}} \cos\tilde{f}x, &
\text{for}\;\;x\in[\alpha,\alpha+l],\\
 -C_5 {\tilde{f}} \sin\tilde{f}x+C_6 {\tilde{f}} \cos\tilde{f}x, &
\text{for}\;\;x\in[\alpha+l,L].
\end{array}
\right.
\end{equation*}

We first consider the case where the Dirichlet boundary condition is applied at $x=0$ and $x=L$.
By the boundary conditions of $\varphi_1$ at $0$ and $L$, we have
$$
\begin{array}{l}
C_3=0,\quad
C_5 \cos {\tilde{g}}L+C_6 \sin {\tilde{g}}L=0,
\end{array}
$$
which imply $C_6=-\cot {\tilde{g}}L \cdot C_5$.
Using the continuous differentiability of $\varphi_1$ at $\alpha$ and $\alpha+l$, we obtain
\begin{equation}\label{RR1H3apdd}
C_4\sin {\tilde{g}}\alpha=C_1\cos\tilde{f}\alpha+C_2\sin\tilde{f}\alpha,
\end{equation}
\begin{equation}\label{RR3H3apdd}
C_4 {\tilde{g}} \cos {\tilde{g}}\alpha=-C_1 {\tilde{f}}\sin\tilde{f}\alpha+C_2 {\tilde{f}} \cos\tilde{f}\alpha,
\end{equation}
\begin{equation}\label{RR2H3apdd}
C_5\cos {\tilde{g}}(\alpha+l)-\cot {\tilde{g}}L\cdot C_5 \sin {\tilde{g}}(\alpha+l)=C_1\cos\tilde{f}(\alpha+l)+C_2\sin\tilde{f}(\alpha+l),
\end{equation}
\begin{equation}\label{RR4H3apdd}
-C_5\tilde{g}\sin {\tilde{g}}(\alpha+l)-\cot {\tilde{g}}L\cdot C_5 {\tilde{g}} \cos {\tilde{g}}(\alpha+l)
=-C_1 {\tilde{f}}\sin\tilde{f}(\alpha+l)+C_2 {\tilde{f}}\cos\tilde{f}(\alpha+l).
\end{equation}
Then
$\eqref{RR1H3apdd}\times\eqref{RR2H3apdd}+\dfrac{\eqref{RR3H3apdd}\times\eqref{RR4H3apdd}}{\tilde{f}^2}$
leads to
\begin{equation}\label{RR01H3apdd}
\begin{array}{rl}
&C_4 C_5\Big[\sin {\tilde{g}}\alpha(\cos {\tilde{g}}(\alpha+l)-\cot {\tilde{g}}L\cdot \sin {\tilde{g}}(\alpha+l))-\D \frac{{\tilde{g}} \cos {\tilde{g}}\alpha(\tilde{g}\sin {\tilde{g}}(\alpha+l)-\cot {\tilde{g}}L\cdot  {\tilde{g}} \cos {\tilde{g}}(\alpha+l))}{{\tilde{f}}^2}\Big]\\
& =(C_1^2+C_2^2)\cos {\tilde{f}}l
\end{array}
\end{equation}
and $\eqref{RR1H3apdd}\times\eqref{RR4H3apdd}-\eqref{RR3H3apdd}\times\eqref{RR2H3apdd}$ leads to
\begin{equation}\label{RR02H3apdd}
\begin{array}{rl}
&-C_4 C_5(\sin {\tilde{g}}\alpha(\tilde{g}\sin {\tilde{g}}(\alpha+l)+\cot {\tilde{g}}L\cdot  {\tilde{g}} \cos {\tilde{g}}(\alpha+l))+{\tilde{g}}\cos {\tilde{g}}\alpha (\cos {\tilde{g}}(\alpha+l)-\cot {\tilde{g}}L\cdot C_5 \sin {\tilde{g}}(\alpha+l)))\\
& =-(C_1^2+C_2^2){\tilde{f}} \sin {\tilde{f}}l.
\end{array}
\end{equation}
By simplifying (\ref{RR01H3apdd}) and (\ref{RR02H3apdd}), we see that (\ref{DDtanequgg0}) holds.
Differentiating (\ref{DDtanequgg0})
 with respect $\alpha$ on two sides, we get (\ref{DDlambdaalphagg0}).

We next consider the case where the Neumann boundary condition is applied at $x=0$ and the Dirichlet boundary condition is applied at $x=L$.
In view of the boundary conditions of $\varphi_1$ at $0$ and $L$, we find $$
\begin{array}{l}
C_4=0,\quad
C_5 \cos {\tilde{g}}L+C_6 \sin {\tilde{g}}L=0,
\end{array}
$$
which imply $C_6=-\cot {\tilde{g}}L \cdot C_5$.
Thanks to the continuous differentiability of $\varphi_1$ at $\alpha$ and $\alpha+l$, we see
\begin{equation}\label{RR1H3apnd}
C_3\cos {\tilde{g}}\alpha=C_1\cos\tilde{f}\alpha+C_2\sin\tilde{f}\alpha,
\end{equation}
\begin{equation}\label{RR3H3apnd}
-C_3 {\tilde{g}} \sin {\tilde{g}}\alpha=-C_1 {\tilde{f}}\sin\tilde{f}\alpha+C_2 {\tilde{f}} \cos\tilde{f}\alpha,
\end{equation}
\begin{equation}\label{RR2H3apnd}
C_5\cos {\tilde{g}}(\alpha+l)-\cot {\tilde{g}}L\cdot C_5 \sin {\tilde{g}}(\alpha+l)=C_1\cos\tilde{f}(\alpha+l)+C_2\sin\tilde{f}(\alpha+l),
\end{equation}
\begin{equation}\label{RR4H3apnd}
-C_5\tilde{g}\sin {\tilde{g}}(\alpha+l)-\cot {\tilde{g}}L\cdot C_5 {\tilde{g}} \cos {\tilde{g}}(\alpha+l)
=-C_1 {\tilde{f}}\sin\tilde{f}(\alpha+l)+C_2 {\tilde{f}}\cos\tilde{f}(\alpha+l).
\end{equation}
Then
$\eqref{RR1H3apnd}\times\eqref{RR2H3apnd}+\dfrac{\eqref{RR3H3apnd}\times\eqref{RR4H3apnd}}{\tilde{f}^2}$
leads to
\begin{equation}\label{RR01H3apnd}
\begin{array}{rl}
&C_3 C_5\Big[\cos {\tilde{g}}\alpha(\cos {\tilde{g}}(\alpha+l)-\cot {\tilde{g}}L\cdot \sin {\tilde{g}}(\alpha+l))+\D\frac{{\tilde{g}} \sin {\tilde{g}}\alpha(\tilde{g}\sin {\tilde{g}}(\alpha+l)+\cot {\tilde{g}}L\cdot  {\tilde{g}} \cos {\tilde{g}}(\alpha+l))}{{\tilde{f}}^2}\Big]\\
& =(C_1^2+C_2^2)\cos {\tilde{f}}l,
\end{array}
\end{equation}
and $\eqref{RR1H3apnd}\times\eqref{RR4H3apnd}-\eqref{RR3H3apnd}\times\eqref{RR2H3apnd}$ leads to
\begin{equation}\label{RR02H3apnd}
\begin{array}{rl}
&C_3 C_5(-\cos {\tilde{g}}\alpha(\tilde{g}\sin {\tilde{g}}(\alpha+l)-\cot {\tilde{g}}L\cdot  {\tilde{g}} \cos {\tilde{g}}(\alpha+l))+{\tilde{g}}\sin {\tilde{g}}\alpha (\cos {\tilde{g}}(\alpha+l)-\cot {\tilde{g}}L\cdot C_5 \sin {\tilde{g}}(\alpha+l)))\\
& =-(C_1^2+C_2^2){\tilde{f}} \sin {\tilde{f}}l.
\end{array}
\end{equation}
By simplifying (\ref{RR01H3apnd}) and (\ref{RR02H3apnd}), we obtain (\ref{NDtanequgg0}).
Differentiating (\ref{NDtanequgg0})
 with respect $\alpha$ on two sides gives rise to (\ref{NDlambdaalphagg0}).

\section*{Acknowledgments} Y. Jin is supported in part by Simons Collaboration Grants for Mathematicians 713985, R. Peng is supported in part by NSFC grant 11671175, and J. F. Wang is supported in part by NSFC grant 11971135.

\end{document}